\documentclass[11pt,parskip=half]{scrartcl}
\usepackage[a4paper,hmargin={2.5cm,2.5cm},vmargin={2.5cm,2.5cm},heightrounded, marginparwidth=2.2cm, marginparsep=0.1cm]{geometry}

\usepackage[reqno]{amsmath}
\usepackage{amssymb,amsthm}
\usepackage{stmaryrd}
\usepackage{color}

\usepackage{tikz-cd}

\theoremstyle{plain}
\newtheorem{theorem}{Theorem}[section]
\newtheorem{lemma}[theorem]{Lemma}
\newtheorem{proposition}[theorem]{Proposition}
\newtheorem{corollary}[theorem]{Corollary}

\theoremstyle{definition}

\newtheorem{example}[theorem]{Example}
\newtheorem{examples}[theorem]{Examples}

\theoremstyle{remark}
\newtheorem{remark}[theorem]{Remark}

\counterwithin*{equation}{section}

\RequirePackage[shortlabels]{enumitem}
\newlist{tfae}{enumerate}{1}%
\setlist[tfae,1]{label=(\roman*)}

\setenumerate[1]{label={(\arabic*)}}

\newcommand{\catfont}[1]{\mathsf{#1}}

\newcommand{\Set}{\catfont{Set}}
\newcommand{\Top}{\catfont{Top}}
\newcommand{\Ord}{\catfont{Ord}}

\newcommand{\catC}{\catfont{C}}
\renewcommand{\epsilon}{\varepsilon}

\newcommand{\LaxComma}[2]{#1{\,\Downarrow\,}#2}
\newcommand{\TopX}{\LaxComma{\Top}{X}}

\newcommand{\SetX}{\LaxComma{\Set}{X}}
\newcommand{\Comma}[2]{#1{\,\downarrow\,}#2}

\DeclareMathOperator{\Fam}{\mathsf{Fam}}

\DeclareMathOperator{\Lan}{\mathsf{Lan}}

\newcommand{\doo}[1]{\overset{\centerdot}{#1}}

\newcommand{\ev}{\mathrm{ev}}

\newcommand{\id}{\mathrm{id}}

\newcommand{\lcm}{\mathrm{lcm}}

\newcommand{\bkN}{\mathbb{N}}
\newcommand{\bkS}{\mathbb{S}}
\newcommand{\bkZ}{\mathbb{Z}}

\usepackage[auth-lg]{authblk}

\title{Topological lax comma categories}

\author[1]{Maria Manuel Clementino}

\affil[1]{CMUC, Department of Mathematics, University of Coimbra, 3000-143
Coimbra, Portugal, \texttt{mmc@mat.uc.pt}}

\author[2]{Dirk Hofmann}

\author[3]{Rui Prezado}

\affil[2,3]{Center for Research and Development in Mathematics and
Applications (CIDMA), Department of Mathematics, University of Aveiro,
3810-193 Aveiro, Portugal, \texttt{dirk@ua.pt}\hspace{3mm}\texttt{ruiprezado@ua.pt}}

\usepackage[hypertexnames=false]{hyperref}

\hypersetup{
  colorlinks = true,
  citecolor= [rgb]{0,0.75,0}, 
  urlcolor=[rgb]{0.4,0,0.4}, 
  linkcolor=[rgb]{0.8,0,0} 
}

\begin{document}

\maketitle

\begin{abstract}
  This paper investigates the interplay between properties of a topological
  space \(X\), in particular of its natural order, and properties of the lax
  comma category \(\TopX\), where \(\Top\) denotes the category of topological
  spaces and continuous maps. Namely, it is shown that, whenever \(X\) is a
  topological \(\bigwedge\)-semilattice, the canonical forgetful functor
  \(\TopX\to\Top\) is topological, preserves and reflects exponentials, and
  preserves effective descent morphisms. Moreover, under additional conditions
  on \(X\), a characterisation of effective descent morphisms is obtained.
  \vspace{3mm}


\end{abstract}

\paragraph{Acknowledgements.}
We thank Matías Menni for calling our attention to the recent paper
\cite{AMS25} and for raising us the question under which conditions the
canonical forgetful functor \(\TopX\to\Top\) is topological.

\paragraph{Funding}
This work was partially supported by the Centre for Mathematics of the
University of Coimbra UID/MAT/00324/2025, funded by the Portuguese Government
through the Portuguese Foundation for Science and Technology/FCT, and by CIDMA
under the Portuguese Foundation for Science and Technology (FCT,
https://ror.org/00snfqn58)  Multi-Annual Financing Program for R\&D Units,
grants UID/4106/2025 and UID/PRR/4106/2025.

\paragraph{Author Contributions Statement}
The authors declare that MMC, DH and RP made equal contributions to all stages
of the development of this work.

{\let\thefootnote\relax\footnote{Corresponding author: Rui Prezado}}

\section*{Introduction}

The recent study of the lax comma categories \(\LaxComma{\Ord}{X}\) and
\(\LaxComma{\catfont{Cat}}{X}\) in \cite{CN23, CNP24, CJ24, CP25}, and the
implicit use of \(\LaxComma{\Top}{\bkN}\) by \cite{AMS25}, led the authors of
this paper to investigate the behaviour of the lax comma category \(\TopX\) we
describe below, namely its (co)completeness, exponentiability and descent. As
expected, this behaviour depends essentially on the interplay between the
order and the topology on \(X\), and leads to the study of interesting
properties of this order.

For a topological space \(X\), we consider the \emph{natural order}
 on \(X\) \cite{Hoc69}, with \(x\leq y\) whenever the principal filter
\(\doo{x}\) converges to \(y\): \(\doo{x}\leadsto y\); equivalently, whenever every
open neighbourhood of \(y\) contains \(x\): \(\mathcal{O}(y)\subseteq
\mathcal{O}(x)\). This relation is reflexive and transitive; it is
anti-symmetric if and only if \(X\) is a T0-space. Every continuous map \(f\colon X\to Y\) is a monotone map with respect to the natural orders on \(X\) and \(Y\), and therefore this construction defines a
functor \(\Top\to\Ord\) from \(\Top\) onto the category \(\Ord\) of ordered
sets and monotone maps, and allows us to consider \(\Top\) as an ordered
category where, for continuous maps \(f,g \colon X\to Y\), \(f\leq g\)
if \(f(x)\leq g(x)\), for all \(x\in X\). We note that the order
relation on a space considered in this paper is dual to the specialisation
order \cite{Sco72} used by many authors. We find it convenient to
think of the convergence of a topological space as a ``generalised order'',
and therefore take as underlying order of a space the ``point shadow'' of the
convergence relation. This intuition makes also a better connection with the
results of \cite{CN23, CJ24, CP25} on lax comma categories of ordered sets.

The category \(\TopX\), called \emph{lax comma category} over \(X\), has as objects pairs \((A,\alpha)\), where \(A\) is a topological space and \(\alpha \colon A\to X\) is a continuous map, and as morphisms \((A,\alpha) \to (B,\beta)\) continuous maps \(f \colon A\to B\) such that \(\alpha\leq \beta\cdot f\):
\begin{displaymath}
  \begin{tikzcd}
    A \ar{rr}{f}[swap, inner sep=1em]{\leq}\ar{dr}[swap]{\alpha}
    && B\ar{dl}{\beta}\\
    & X
  \end{tikzcd}
\end{displaymath}
It has as a (wide) subcategory the comma category \(\Comma{\Top}{X}\), which
has the same objects as \(\TopX\) but a morphism \(f\colon
(A,\alpha)\to(B,\beta)\) must satisfy \(\alpha=\beta\cdot f\). It is
well-known that the category \(\Comma{\Top}{X}\) is complete and cocomplete,
and its canonical functor \(\Comma{\Top}{X}\to\Top\) preserves equalisers and
colimits, but in general not products, nor exponentials; moreover, it
trivially preserves and reflects (effective) descent morphisms.

If \(X\) is a T1-space, then its natural order is discrete, and therefore the lax comma category \(\TopX\) coincides with the comma category \(\Comma{\Top}{X}\). Furthermore, for any \(X\), if we denote by \(\eta_X\colon X\to RX\) its T0-reflection, where \(\eta_X(x)=\eta_X(y)\) if and only if \(x\leq y\) and \(y\leq x\), and \(RX\) has the quotient topology, it is easily checked that the functor \(\TopX\to\LaxComma{\Top}{RX}\) defined by composing with \(\eta_X\) is an equivalence of categories making the diagram
\begin{displaymath}
  \begin{tikzcd}
    \TopX \ar{rr}{}[swap, inner sep=1em]{}\ar{dr}[swap]{G}
    && \Top\Downarrow RX\ar{dl}{G}\\
    & \Top
  \end{tikzcd}
\end{displaymath}
commute, where \(G\) denotes the canonical forgetful functor(s). Therefore, for simplicity, we may reduce our study to categories \(\TopX\) with \(X\) a T0-space:
\begin{center}
  \emph{From now on we assume that \(X\) is a topological T0-space}.
\end{center}

In Section \ref{sec:topologicity} we investigate the existence of limits and colimits in \(\TopX\), starting by showing that sums and equalisers always exist, and that \(\TopX\) is an (infinitely) extensive category. It is also shown that the existence of coequalisers and their preservation by \(G\colon\TopX\to\Top\) is equivalent to the existence and preservation of products by \(G\); it is in fact equivalent to \(G\) being a topological functor, and equivalent to \(X\) being sober and a topological \(\wedge\)-semilattice. These properties of \(X\) can be formulated in several different ways, as stated in Theorem \ref{d:thm:3}, and in particular reveal \(X\) as an algebra for the lower Vietoris monad on \(\Top\).

In Section \ref{sec:exponentiability} we study exponentiability in \(\TopX\), making use of the Special Adjoint Functor Theorem. In particular we generalise results obtained (using different techniques) in \cite{Nie06}, showing that, if \(X\) is a topological Heyting \(\wedge\)-semilattice, then the forgetful functor \(G\) preserves and reflects exponentiable objects and exponentials; moreover, reflection of exponentiable objects by \(G\) is equivalent to \(X\) being a topological Heyting \(\wedge\)-semilattice.

In the last section we study effective descent morphisms in \(\TopX\). Here we
follow the strategy of \cite{CJ24, CP25}, considering the pseudopullback
\begin{displaymath}
  \begin{tikzcd}
    \Top\times_{\Set}\Fam(X) %
    \ar{r}{\rho_{2}} %
    \ar{d}[swap]{\rho_{1}} %
    & \Fam(X) %
    \ar{d}{} \\
    \Top %
    \ar{r}[swap]{} %
    & \Set %
  \end{tikzcd}
\end{displaymath}
and the canonical functor
\begin{displaymath}
  H \colon\TopX \longrightarrow \Top\times_{\Set}\Fam(X).
\end{displaymath}
induced by the corresponding forgetful functors. This leads us to obtain, in
Theorems \ref{d:thm:7} and \ref{d:thm:8}, characterizations of effective
descent morphisms in \(\TopX\), for suitable \( X \). We point out that the category \(\Fam(X)\) is a lax comma category -- \(\SetX\) -- but we chose to use here, as in \cite{CP25}, the more familiar notation \(\Fam(X)\).

\section{The category \(\TopX\)}
\label{sec:topologicity}

In this section we investigate basic properties of \(\TopX\) and their relationship with properties of the base space \(X\).
As already observed by \cite{Nie06}, the canonical forgetful functor \(G\colon\TopX\to\Top\) has a left adjoint if and only if \(X\) has a bottom element \(\bot\); in this case, the left adjoint \(L\colon\Top\to\TopX\) is also a right inverse; it sends a space \(A\) to the pair consisting of \(A\) together with the constant map
\begin{align*}
  A & \longrightarrow X\\
  a & \longmapsto \bot
\end{align*}
and leaves continuous maps unchanged. Hence \(G\) is a right adjoint if and only if it is a \emph{rali} (=right adjoint left inverse).
On the other hand, \(G\colon\TopX\to\Top\) is a \emph{lali} (left adjoint left inverse) if and only if \(X\) has a top element \(\top\); here its right adjoint right inverse \(\Top\to\TopX\) sends a space \(A\) to the pair consisting of \(A\) together with the constant map \(A\to X,\,a\mapsto\top\). Note that \(G\) may be a left adjoint without being a \emph{lali}; this is for instance the case when \(X\) is a T1-space, since the forgetful functor \(\Comma{\Top}{X}\to\Top\) is always a left adjoint.
Furthermore, the following properties are easy to verify:
\begin{proposition}
  \label{d:prop:1}
  \begin{enumerate}
  \item The category \(\TopX\) has sums: for a family \((A_i,\alpha_i)_{i\in I}\) of objects of \(\TopX\), its sum in \(\TopX\) is given by
    \begin{displaymath}
      \left(\Big(\sum_{i\in I}A_i,\alpha\Big),\Big(\iota_j\colon A_j\to\sum_{i\in I}A_i\Big)_{j\in I}\right),
    \end{displaymath}
    where \((\sum_{i\in I}A_i,(\iota_j)_{j\in I})\) is the sum of $(A_i)_{i\in I}$ in \(\Top\), and \(\alpha\colon \sum_{i\in I}A_i\to X\) is the continuous map induced in the topological sum by the family \( (\alpha_i\colon A_i\to X)_{i\in I}\).
  \item\label{dois} The category \(\TopX\) has equalisers: for a parallel pair \(f,g \colon (A,\alpha) \to (B,\beta)\) of arrows in \(\TopX\), its equaliser in \(\TopX\) is given by \(e \colon (E,\alpha\cdot e)\to(A,\alpha)\) where
    \begin{displaymath}
      \begin{tikzcd}
        E\ar{r}{e} %
        & A\ar[shift left=3pt]{r}{f} \ar[shift right=3pt]{r}[swap]{g} %
        & B
      \end{tikzcd}
    \end{displaymath}
    is an equaliser in \(\Top\).
  \end{enumerate}
\end{proposition}

Recall that a category \(\catfont{C}\) with sums is called \emph{(infinitely) extensive} (see \cite{CLW93}) whenever, for every small family \((C_i)_{i\in I}\) of objects of \(\catfont{C}\), the canonical functor
\begin{displaymath}
  \sum \colon
  \prod_{i\in I}(\Comma{\catfont{C}}{C_i}) \longrightarrow
  \Comma{\catfont{C}}{\Big(\sum_{i\in I}C_i\Big)}
\end{displaymath}
is an equivalence. The category \(\Top\) is extensive, and this property is inherited by \(\TopX\):

\begin{proposition}
  The category \(\TopX\) is extensive.
\end{proposition}
\begin{proof}
  The assertion follows immediately from the facts that \(\Top\) is extensive and that, for all small families \((A_i)_{i\in I}\) of topological spaces, the canonical map
  \begin{displaymath}
    \prod_{i\in I}\Top(A_i,X) \longrightarrow
    \Top\Big(\sum_{i\in I}A_i,X\Big)
  \end{displaymath}
  is an isomorphism of ordered sets (see also \cite[Theorem~41]{NV23}).
\end{proof}

Being extensive guarantees that a category has ``well-behaved sums''; moreover, extensive categories with finite products are (infinitely) distributive:

\begin{corollary}
  \label{d:cor:3}
  Assume that all binary products exist in \(\TopX\). Then the functor
  \begin{displaymath}
    (A,\alpha)\times - \colon \TopX \longrightarrow\TopX
  \end{displaymath}
  preserves sums.
\end{corollary}

Further important (co)completeness properties of \(\TopX\) are in direct relation with (co)complete\-ness properties of the space \(X\), as we explain next.
We start by considering the existence of coequalisers. We point out that our assumption that the base space $X$ is T0 guarantees uniqueness of infima and suprema.

\begin{lemma}
  Let \(f,g \colon (A,\alpha)\to (B,\beta)\) be arrows in \(\TopX\) and let \(q \colon B\to Q\) be a coequaliser of \(f,g \colon A\to B\) in \(\Top\).
  Then \(q \colon (B,\beta)\to(Q,\gamma)\) is a coequaliser of \(f,g\colon (A,\alpha)\to (B,\beta)\) in \(\TopX\) if and only if \(\gamma \colon Q\to X\) is the left Kan extension \(\Lan_q(\beta)\colon Q\to X\) of \(\beta \colon B\to X\) along \(q \colon B\to Q\).
  \begin{displaymath}
    \begin{tikzcd}[column sep=large,row sep=large]
      A \ar[shift left=3pt]{r}{f} %
        \ar[shift right=3pt]{r}[swap]{g} %
        \ar[rd,phantom,"\leq",pos=0.6,shift left=2.3mm]
        \ar[dr,swap,"\alpha",bend right=25]
      & B
         \ar[d,"\beta" description]
         \ar[r,"q"]
      & Q
        \ar[dl,phantom,"\leq",pos=0.6,shift right=2.3mm]
        \ar[dl,dotted,bend left=25,"\gamma=\Lan_q(\beta)"]\\
      & X
    \end{tikzcd}
  \end{displaymath}
\end{lemma}
\begin{proof}
 See \cite[Theorem~9.3, version 1 of arXiv]{CN24}.
\end{proof}

\begin{corollary}
 The category \(\TopX\) has and the functor \(\TopX\to\Top\) preserves coequalisers if and only if, for every continuous map \(\beta \colon B\to X\) and every quotient map \(q \colon B\to Q\), the left Kan extension \(\Lan_q(\beta)\colon Q\to X\) of \(\beta\) along \(q\) exists.
\end{corollary}

Recall that \(\TopX\to\Top\) is a left adjoint whenever \(X\) has a top element, hence, under this condition, \(\TopX \to \Top\) preserves all existing colimits.
Those spaces admitting left Kan extensions along arbitrary continuous maps are characterised in \cite[Theorem~2.4]{Erk98}:
\begin{theorem}
  \label{d:thm:4}
  The T0-space \(X\) admits left Kan extensions along arbitrary continuous maps if and only if \(X\) is sober, complete with respect to the natural order, and the map \(\wedge \colon X\times X\to X\) is continuous.
\end{theorem}

Below we analyse the conditions for the existence of left Kan extensions of the theorem above more in detail, starting with some notation.
We call a T0-space \(X\) a \emph{topological \(\wedge\)-semilattice} whenever \(X\) is a \(\wedge\)-semilattice with respect to the natural order and the map
\begin{displaymath}
  \wedge \colon X\times X \longrightarrow X
\end{displaymath}
is continuous.
Similarly, we call \(X\) a \emph{topological \(\bigwedge\)-semilattice} whenever the ordered set \(X\) is complete and, for every set \(I\), the map \(\bigwedge \colon X^I\to X\) is continuous.
\begin{remark}
  The space \(X\) has continuous infima of all families indexed by the set \(I\) if and only if the diagonal map \(\Delta_X \colon X\to X^I\) has a right adjoint in the ordered category \(\Top\).
\end{remark}
The following result is due to \cite[Theorem~2.6]{Hof79a} and \cite[Section~6.3]{Sch93}.
\begin{theorem}
  \label{d:thm:1}
  The following assertions are equivalent.
  \begin{tfae}
  \item\label{d:item:7} The topological space \(X\) is a topological \(\bigwedge\)-semilattice.
  \item\label{d:item:6} The topological space \(X\) is sober and a topological \(\wedge\)-semilattice.
  \end{tfae}
\end{theorem}

\begin{proof}
  Regarding \ref{d:item:6}\(\implies\)\ref{d:item:7}, just observe that, by \cite{Wyl81a, Hof79a} (see also \cite[Lemma~II.1.9]{Joh86}), every sober space admits codirected infima with respect to the natural order, and, for every codirected ordered set \((I,\leq)\), the map
  \begin{displaymath}
    X^I\supseteq
    \{(x_i)_{i\in I}\in X^I\mid I\to X,\,i\mapsto x_i,\text{ is monotone}\}
    \xrightarrow{\quad\bigwedge\quad}X
  \end{displaymath}
  is continuous, with the domain considered as a subspace of the product space \(X^I\).
\end{proof}

With respect to the natural order of a space, every open subset is down-closed, likewise, every closed subset is up-closed. Given now an ordered set \(Z\), there are various topologies on the set \(Z\) which induce the given order. Here we list three such topologies (see \cite{GHK+03}):
\begin{enumerate}
\item the \emph{lower topology} on \(Z\), where the closed sets are generated by the sets of the form \({\uparrow}a=\{z\in Z\mid a\leq z\}\), with \(a\in Z\),
\item the \emph{lower Scott topology} on \(Z\), where \(A\subseteq Z\) is closed whenever \(A\) is up-closed and stable under codirected infima, and
\item the \emph{lower Alexandroff topology} on \(Z\), where \(A\subseteq Z\) is closed whenever \(A\) is up-closed.
\end{enumerate}
Clearly, the lower topology is included in the lower Scott topology which in turn is included in the lower Alexandroff topology. Furthermore, the lower topology is the coarsest topology inducing the given order, while the lower Alexandroff topology is the largest such topology. A topological space is called an \emph{Alexandroff space} whenever its topology coincides with the lower Alexandroff topology of its natural order; equivalently, whenever arbitrary unions of closed subsets are closed.

\begin{lemma}\label{m:adjoint}
  Assume that \(X\) is equipped with the lower topology with respect to its natural order.
  \begin{enumerate}
  \item\label{d:item:20} Given a space \(Y\) and a pair of adjoint maps \((X\xrightarrow{f}Y)\dashv(Y\xrightarrow{g}X)\) in \(\Ord\), with respect to the natural order of \(X\) and \(Y\), \(g\) is continuous.
  \item\label{d:item:21} If the natural order on \(X\) is complete, then \(X\) is a topological \(\bigwedge\)-semilattice; in particular, \(X\) is sober.
  \end{enumerate}
\end{lemma}
\begin{proof}
  \ref{d:item:20} Under these conditions, for every \(x\in X\):
  \begin{displaymath}
    g^{-1}({\uparrow} x)
    =\{ y\in Y\mid x\leq g(y)\}=\{y\in Y\mid f(x)\leq y\}=\uparrow f(x).
  \end{displaymath}
  \ref{d:item:21} follows from \ref{d:item:20}, since, for every set \(I\), the map \(\bigwedge \colon X^I\to X\) is right adjoint to the diagonal \(\Delta\colon X\to X^I\).
\end{proof}

\begin{remark}
  Our naming scheme for the various topologies on an ordered set \(Z\)
  reflects the fact that we declare certain down-closed (also called
  \emph{lower}) subsets of \(Z\) to be open. There are also the somehow dual
  topologies on an ordered set \(Z\), namely the upper topology, the upper
  Scott topology and the upper Alexandroff topology which are defined so that
  the upper (Scott, Alexandroff) topology of \(Z\) coincides with the lower
  (Scott, Alexandroff) topology of the dual ordered set \(Z^{\mathrm{op}}\).
  We note that these topologies have the dual order of \(Z\) as natural order,
  which of course means that the specialisation order of these topologies is
  the given order. For instance, a topological space carries the lower Scott
  topology with respect to the natural order if and only if it carries the
  upper Scott topology with respect to the specialisation order.
\end{remark}

\cite{Sco72} has characterised the injective T0-spaces as precisely the
continuous lattices with respect to the specialisation order equipped with the
upper Scott topology; hence, the op-continuous lattices with respect to the
natural order equipped with the lower Scott topology. We recall that an
\emph{op-continuous lattice} is a complete lattice \(Z\) so that every element
is the meet of its way above elements -- herein, we say \(x\) is \emph{way
above} \(y\), denoted \(x\gg y\), if, whenever \(S\) is a codirected subset of
\(Z\) with \(y\geq\bigwedge S\), there exists \(s\in S\) such that \(x\geq
s\). We recall that a topological (T0-)space \(Z\) is said to be
\emph{injective} if, for every embedding \(f\colon A\to B\) and every
continuous map \(g\colon A\to Z\), there exists a continuous map
\(\overline{g}\colon B\to Z\) making the following diagram commute
\[
  \begin{tikzcd}
    A \ar[rd,swap,"g"] \ar[rr,"f"]
    && B \ar[ld,"\overline{g}"] \\
    & Z
  \end{tikzcd}
\]

\begin{lemma}
  Assume that \(X\) is an injective space. Then \(X\) is a topological \(\bigwedge\)-semilattice (with respect to the natural order). Moreover, the map \(\vee \colon X\times X\to X\) is continuous.
\end{lemma}

\begin{proof}
  In \cite[Proposition~2.7]{Sco72} it is shown that, for any injective space \(X\), both maps \(\wedge\colon X\times X\to X\) and \(\vee\colon X\times X\to X\) are continuous. Moreover, every injective space is a retract of a product of copies of the Sierpinski space, hence it is sober. By Theorem~\ref{d:thm:1}, \(X\) is a topological \(\bigwedge\)-semilattice as claimed.
\end{proof}

\begin{examples}\label{mm:examples1}
  \begin{enumerate}
  \item Let \(X=\bkS\) be the Sierpinski space, with \(\{1\}\) the only non-trivial closed subset. Clearly \(0\leq 1\) for the natural order, which is both a continuous and an op-continuous lattice. Here, as in any finite topological space, the three topologies considered above coincide.

    The category \(\LaxComma{\Top}{\bkS}\) may be equivalently described as the category with objects pairs \((A,A_0)\) where \(A\) is a topological space and \(A_0\) is a closed subset of \(A\), and with morphisms \(f \colon(A,A_0)\to(B,B_0)\) those continuous maps \(f \colon A\to B\) with \(f(a)\in B_0\) for all \(a\in A_0\), that is, \(f\) restricts to \(f_0 \colon A_0\to B_0\).
    \begin{displaymath}
      \begin{tikzcd}
        A_{0} %
        \ar{r}{f_{0}} %
        \ar[>->]{d}{} %
        & B_{0} %
        \ar[>->]{d}{} \\
        A %
        \ar{r}[swap]{f} %
        & B %
      \end{tikzcd}
    \end{displaymath}
  \item Consider \(X=\mathbb{N}\) with the order \(\preceq\) defined by \(n\preceq m\) if \(m\) divides \(n\) (and we write \(m\mid n\)). It is a complete lattice, with \(\top=1\), \(\bot=0\), and \(\bigvee_{i\in I}m_i=\gcd(m_i)\), \(\bigwedge_{i\in I} m_i=\lcm(m_i)\). Then \(X\) is an op-continuous lattice but it is not a continuous lattice: it is easily seen that if \(p^\alpha\mid m\), for \(p\) a prime number, then \(p^\alpha\gg m\), while, if \(m\) is not a prime number, then \(n\ll m\) only if \(n=0\). Again the lower topology coincides with the lower Scott topology. We note that in \cite{AMS25} the category \(\LaxComma{\Top}{\bkN}\), called there the category of \(a\)-spaces and \(a\)-maps, is used to obtain a generalization of Yosida-duality to norm-complete \(l\)-groups. The dual category is the full subcategory of \(\LaxComma{\Top}{\bkN}\) defined by compact Hausdorff spaces \(A\) with certain continuous maps \(A\to\bkN\).

    We point out that with the lower Alexandroff topology on \(\bkN\) we do not obtain a topological \(\bigwedge\)-semilattice: this space is not sober.

  \item Let \(X\) be the interval \([0,1]\) with the lower topology for the natural order on \([0,1]\); that is, \(U\subseteq [0,1]\) is closed if and only if \(U=\varnothing\) or \(U=[x,1]\) for some \(x\in [0,1]\). Clearly, \([0,1]\) is both a continuous and an op-continuous lattice, and closed subsets are stable under (codirected) infima, hence \(X\) carries the lower Scott topology, and, consequently, \(X\) is an injective space. We point out that this topology does not coincide with the lower Alexandroff topology: the intervals \(]x,1]\) are up-closed but not closed for the lower topology. Again, \[[0,1]\] with the lower Alexandroff topology is not a sober space.

      For \(X= [0,1]\) with the lower topology (= Scott topology), the category
      \(\TopX\) admits a similar description as for \(X=\bkS\). Firstly, for
      each set \(A\), there is a bijection
    \begin{displaymath}
      \{\alpha \colon A\to [0,1]\}
      \longrightarrow
      \{(A_u)_{u\in [0,1]}\mid \text{for all \(u\in [0,1]\): }
      A_u\subseteq A\;\text{and}\; A_u=\bigcap_{v<u}A_v\}
    \end{displaymath}
    sending \(\alpha \colon A\to [0,1]\) to the family \((\alpha^{-1}({\uparrow} u))_{u\in [0,1]}\). The inverse of the function above sends the family \((A_u)_{u\in [0,1]}\) to the map
    \begin{displaymath}
      \alpha \colon A \longrightarrow [0,1],\,
      a \longmapsto \sup\{u\in [0,1]\mid a\in A_u\}.
    \end{displaymath}
    Furthermore, if \(A\) is a topological space, then a map \(\alpha \colon A\to [0,1]\) is continuous if and only if the corresponding family \((A_u)_{u\in [0,1]}\) consists only of closed subsets of \(A\). Given also continuous maps \(\beta \colon B\to [0,1]\) (with corresponding family \((B_u)_{u\in [0,1]}\)) and \(f \colon A \to B\), then \(\alpha\leq\beta\cdot f\) if and only if, for all \(u\in [0,1]\),
    \begin{equation}
      \label{d:eq:5}
      \forall a\in A\,.\,(a\in A_u \implies f(a)\in B_u);
    \end{equation}
    that is, \(f\) restricts to a map \(f_u \colon A_u\to B_u\).
    \begin{displaymath}
      \begin{tikzcd} 
        A_{u} %
        \ar{r}{f_{u}} %
        \ar[>->]{d}[swap]{} %
        & B_{u} %
        \ar[>->]{d}{} \\
        A %
        \ar{r}[swap]{f} %
        & B %
      \end{tikzcd}
    \end{displaymath}
    Hence, we may consider \(\LaxComma{\Top}{[0,1]}\) as the category with objects pairs \((A,(A_u)_{u\in [0,1]})\) consisting of a topological space \(A\) and a family \((A_u)_{u\in [0,1]}\) of closed subsets of \(A\) so that, for all \(u\in [0,1]\), \(A_u=\bigcap_{v<u}A_v\), and with morphisms \(f \colon (A,(A_u)_{u\in [0,1]})\to (B,(B_u)_{u\in [0,1]})\) those continuous maps \(f \colon A \to B\) satisfying \eqref{d:eq:5} for all \(u\in [0,1]\).
  \item Consider the complete lattice \(X=\bkZ_-^{\infty}=\{k\in\bkZ\mid k\leq 0\}\cup\{-\infty\}\), with the usual order. Then it is easy to check that \(\bkZ_-^{\infty}\) is both a continuous and an op-continuous lattice. The way above relation is given by: \(k\gg j\) if \(k\geq j\) and \(k\gg -\infty\) for every \(k\in\bkZ_-\). Again, its lower topology coincides with the lower Scott topology but it is weaker than its lower Alexandroff topology: \(\bkZ_-\) is up-closed but not closed in the lower topology.
  \item Now let \(X=\bkZ_+^{\infty}=\{k\in\bkZ\mid k\geq 0\}\cup\{+\infty\}\), with the usual order. It is again, of course, both a continuous and an op-continuous lattice. However the situation here is different from the previous example; indeed, every non-empty up-set has a minimum, hence the lower topology and the lower Alexandroff topology coincide.
  \end{enumerate}
\end{examples}

\begin{remark}
  \label{d:rem:1}
  Assume now that \(X\) is an injective space. For an embedding \(q \colon B\to C\) and a continuous map \(\beta \colon B\to X\), as depicted in \eqref{r:eq:lkan}:
  \begin{equation}
    \label{r:eq:lkan}
    \begin{tikzcd}
      B \ar[rd,swap,"\beta"]
        \ar[rr,"q"]
        && C \ar[ld,dashed,"\Lan_q\beta"] \\
      & X
    \end{tikzcd}
  \end{equation}
  the extension \(\Lan_q\beta \colon C\to X\) of \(\beta\) along \(q\) is given by
  \begin{equation}
    \label{d:eq:1}
    c \longmapsto
    \bigwedge_{V\in \mathcal{O}(c)}\bigvee_{b\in q^{-1}(V)}\beta(b).
  \end{equation}
  By considering \(q=1_B \colon B\to B\), we obtain in particular that, for every \(b\in B\),
  \begin{displaymath}
    \beta(b)=
    \bigwedge_{V\in \mathcal{O}(b)}\bigvee_{b'\in V}\beta(b').
  \end{displaymath}
  Furthermore, \cite[Subsection~2.6]{Esc98} observed that the formula \eqref{d:eq:1} also defines the left Kan extension \(\Lan_q(\beta)\colon C\to X\) of \(\beta\) along an arbitrary continuous map \(q \colon B\to C\).
\end{remark}

Next we will expose a connection between topological \(\wedge\)-semilattices and two important constructions of topological spaces: the lower Vietoris space and the ultrafilter space.

The lower Vietoris functor $V\colon\Top \to\Top$ is defined by \(VZ=\{A\subseteq Z\mid A\text{ closed}\}\) with the ``hit topology'', that is, the topology generated by the sets \(\{A\in VZ\mid A\cap W\neq\varnothing\}\)  for all \(W\) open in \(Z\), and by \(Vf(A)=\overline{f[A]}\).
Furthermore, this functor is part of a monad on \(\Top\), with unit \(e \colon 1\to V\) defined by \(e_Z\colon Z\to VZ,\,z\mapsto\overline{\{z\}}\) and multiplication \(m \colon VV\to V\) given by union.
This construction goes back to the work \cite{Vie22}; for more information we refer to \cite{Mic51, Sch93, Sch93a}.
We note that \(A\leq B\) in the natural order of \(VZ\) if and only if \(A\supseteq B\), for all \(A,B\in VZ\).
Furthermore, a sub-basis for the closed subsets of \(VX\) is given by the sets
\begin{displaymath}
  \{A\in VZ\mid B\supseteq A\},
\end{displaymath}
for \(B\subseteq Z\) closed; that is, the lower Vietoris topology on \(VZ\) is the lower topology with respect to the containment order on \(VX\).

For a topological space \(Z\), we consider the set \(UZ\) of all ultrafilters on the set \(Z\) equipped with the topology generated by the basic open sets \(\{\mathfrak{x}\in UZ\mid B\in\mathfrak{x}\}\), with \(B\) open in \(Z\).
This construction extends the ultrafilter monad \((U,m,e)\) on \(\Set\) to a monad on \(\Top\); see \cite{HST14}, for instance. In particular, the unit
\begin{displaymath}
  e_Z \colon Z \longrightarrow UZ,\,z\longmapsto\doo{z}
\end{displaymath}
is continuous.
The natural order of the space \(UZ\) is described by
\begin{align*}
  \mathfrak{x}\leq \mathfrak{y}
  &\iff \forall B\subseteq Z\text{ open}\,.\,(B\in \mathfrak{y}\implies B\in \mathfrak{x})\\
  &\iff \forall A\subseteq Z\text{ closed}\,.\,(A\in \mathfrak{x}\implies A\in \mathfrak{y})\\
  &\iff \forall A\in \mathfrak{x}\,.\,\overline{A}\in \mathfrak{y},
\end{align*}
for all \(\mathfrak{x},\mathfrak{y}\in UZ\).
In particular, considering above \(\mathfrak{y}=\doo{z}\) the principal ultrafilter generated by \(z\in Z\), then \(\mathfrak{x}\leq\doo{z}\) if and only if \(\mathfrak{x}\leadsto z\).
For a topological space \(Z\) and a map \(\sigma \colon UZ\to Z\), one has \(e_Z\dashv\sigma\) in \(\Ord\) if and only if, for all \(z\in Z\) and \(\mathfrak{x}\in UX\),
\begin{displaymath}
  \mathfrak{x}\leq \doo{z}\iff \sigma(\mathfrak{x})\leq z.
\end{displaymath}
Therefore the monotone map \(e_Z \colon Z\to UZ\) has a left adjoint in \(\Ord\) if and only if every ultrafilter \(\mathfrak{x}\) on \(Z\) has a smallest convergence point with respect to the natural order of \(Z\). Below we collect some elementary properties of such spaces.

\begin{lemma}
  \begin{enumerate}
  \item Assume that \(e_Z \colon Z\to UZ\) has a left adjoint \(\sigma \colon UZ\to Z\) in \(\Ord\). Then also \(e_{Z\times Z} \colon Z\times Z\to U(Z\times Z)\) has a left adjoint in \(\Ord\), given by
    \begin{displaymath}
      U(Z\times Z)\xrightarrow{\qquad}UZ\times UZ
      \xrightarrow{\quad\sigma\times\sigma\quad}
      Z\times Z.
    \end{displaymath}
  \item Let \(f \colon Y\to Z\) be a continuous map between T0-spaces and assume that both \(Y\) and \(Z\) have smallest convergence points of ultrafilters. If \(f\) has a right adjoint in \(\Top\), then \(f\) preserves smallest convergence points of ultrafilters.
  \end{enumerate}
\end{lemma}
\begin{proof}
  The first assertion follows immediately from the fact that convergence in \(Z\times Z\) is component-wise. Regarding the second assertion, let \(g \colon Z\to Y\) be the right adjoint of \(f\) in \(\Top\). Since the diagram
  \begin{displaymath}
    \begin{tikzcd}
      Z %
      \ar{r}{g} %
      \ar{d}[swap]{e_{Z}} %
      & Y %
      \ar{d}{e_{Y}} \\
      UZ %
      \ar{r}[swap]{Ug} %
      & UY %
    \end{tikzcd}
  \end{displaymath}
  of right adjoints commutes, the diagram of the corresponding left adjoints in \(\Ord\) commutes as well.
\end{proof}

The algebras for the lower Vietoris monad are characterised as the ``topological complete join-semilattices'' (with respect to specialisation order, which is dual to the order considered in this paper) in \cite[Theorem~2.6]{Hof79a} and in \cite[Section~6.3]{Sch93}. The theorem below summarises their results, and the connection with ``smallest convergence points of ultrafilters'' is described in \cite[Proposition~6.4 and Example~6.5]{Hof14}.

\begin{theorem}
  \label{d:thm:3}
  The following assertions are equivalent, for a T0-space \(Z\).
  \begin{tfae}
  \item\label{d:item:4} \(Z\) is an algebra for the lower Vietoris monad on \(\Top\).
  \item \(Z\) is a topological \(\bigwedge\)-semilattice.
  \item\label{d:item:9} \(Z\) is sober and a topological \(\wedge\)-semilattice.
  \item The ordered set \(Z\) is complete and the unit \(Z\to UZ,\,z\mapsto\doo{z}\) has a left adjoint in \(\Ord\).
  \item\label{d:item:5} The ordered set \(Z\) is finitely complete and the unit \(Z\to UZ,\,z\mapsto\doo{z}\) has a left adjoint in \(\Ord\).
  \end{tfae}
\end{theorem}

\begin{remark}
  Regarding the implication \ref{d:item:5}\(\implies\)\ref{d:item:9} in Theorem~\ref{d:thm:3}, the fact that every ultrafilter has a smallest convergence point implies that every irreducible closed set is the closure of some point (because every irreducible closed set is the set of limit points of some ultrafilter), and that the map \(\wedge \colon Z\times Z\to Z\) is continuous (see \cite[Example~6.5]{Hof14}).
\end{remark}

We turn now our attention to the study of limits in \(\TopX\). The following result is essentially \cite[Proposition~3.1]{Nie06}.

\begin{proposition}
  \label{d:prop:2}
  Let \(I\) be a set. Then the following assertions are equivalent:
  \begin{tfae}
  \item\label{d:item:1} With respect to the natural order, \(X\) has all infima of \(I\)-indexed families, and, moreover, the infima map \(\bigwedge \colon X^I\to X\) is continuous.
  \item\label{d:item:2} The category \(\TopX\) has products of families of objects indexed by \(I\), and these products are preserved by \(\TopX\to\Top\).
  \item\label{d:item:3} The category \(\TopX\) has the \(I\)-powers of the object \((X,1_X \colon X\to X)\), and this power is preserved by \(\TopX\to\Top\).
  \end{tfae}
\end{proposition}
\begin{proof}
  To see \ref{d:item:1}\(\implies\)\ref{d:item:2}, consider a family \((A_i,\alpha_i)_{i\in I}\) of objects of \(\TopX\).
  Its product is given by
  \begin{displaymath}
    \Big(\prod_{i\in I}A_i,\alpha\Big)
  \end{displaymath}
  where \(\alpha\) is the composite
  \begin{equation}
    \label{d:eq:2}
    \begin{tikzcd}[column sep=large]
      \displaystyle\prod_{i \in I} A_i \ar{r}{\langle\alpha_i\rangle_{i\in I}}
      & X^I\ar{r}{\bigwedge}
      & X.
    \end{tikzcd}
  \end{equation}
  The implication \ref{d:item:2}\(\implies\)\ref{d:item:3} is trivial.
  Finally, assuming \ref{d:item:3}, we have necessarily
  \begin{displaymath}
    (X,1_X \colon X\to X)^I=(X^I,\bigwedge \colon X^I\to X),
  \end{displaymath}
  proving \ref{d:item:1}.
\end{proof}

\begin{remark}
  In the sequel we denote the composite map \eqref{d:eq:2}  by
  \begin{displaymath}
    \bigsqcap_{i\in I}\alpha_i \colon \prod_{i\in I}A_i
    \longrightarrow X.
  \end{displaymath}
  As usual, in the binary case we write \(\alpha_1\sqcap\alpha_2\) instead of \(\bigsqcap_{i\in\{1,2\}}\alpha_i\).
\end{remark}

As an immediate consequence of Propositions~\ref{d:prop:2} and \ref{d:prop:1}\ref{dois}, and of Theorem~\ref{d:thm:1}, we obtain:

\begin{corollary}
  \label{d:cor:1}
  \begin{enumerate}
  \item The category \(\TopX\) has and \(\TopX\to\Top\) preserves finite limits if and only if \(X\) is a topological \(\wedge\)-semilattice.
  \item The category \(\TopX\) has and \(\TopX\to\Top\) preserves all limits if and only if \(X\) is a topological \(\bigwedge\)-semilattice.
  \item Assume that \(X\) is sober and has a bottom element.
    Then \(\TopX\) is complete if and only if \(\TopX\) is finitely complete.
  \end{enumerate}
\end{corollary}

In fact, more can be said if \(X\) is a topological \(\bigwedge\)-semilattice.

\begin{theorem}
  The canonical forgetful functor \(\TopX\to\Top\) is topological if and only if \(X\) is a topological \(\bigwedge\)-semilattice.
\end{theorem}
\begin{proof}
  Clearly, if \(\TopX\to\Top\) is topological, then \(\TopX\) is complete and \(\TopX\to\Top\) preserves limits. By Corollary~\ref{d:cor:1}, \(X\) is a topological \(\bigwedge\)-semilattice.

  Assume now that \(X\) is a topological \(\bigwedge\)-semilattice.
  Consider a space \(A\) and (possibly large) families \((A_i,\alpha_i)_{i\in I}\) of objects of \(\TopX\) and \((f_i \colon A\to A_i)_{i\in I}\) of continuous maps.
  Then the class
  \begin{displaymath}
    \mathcal{C}=\{\alpha_i\cdot f_i \colon A\to X\mid i\in I\}
  \end{displaymath}
  is actually a set, so that we may consider it as
  \begin{displaymath}
    \mathcal{C}=\{\gamma_j \colon A\to X\mid j\in J\}
  \end{displaymath}
  for a set \(J\).
  By hypothesis, \(\gamma=\bigwedge_{i\in J}\gamma_j \colon A\to X\) is continuous and defines the initial lift \((A,\gamma)\) of the cone \((f_i \colon A\to A_i)_{i\in I}\) with respect to \((A_i,\alpha_i)_{i\in I}\).
\end{proof}

As we have already mentioned, the category \(\TopX\) differs substantially
from its wide subcategory \(\Comma{\Top}{X}\), whenever \(X\) is not a
T1-space. This is particularly visible in the behaviour of products -- as just
shown, under suitable conditions on \(X\), the forgetful functor \(\TopX\to
\Top\) preserves products while \(\Comma{\Top}{X}\to\Top\) does not in general
-- and on the behaviour of exponentiability and descent, as detailed in the
forthcoming sections.

\section{Exponentiability}
\label{sec:exponentiability}

We recall that a category \(\catC\) is said to be \emph{cartesian closed} if it has finite products and, for each \(A\in\catC\), the functor \(A\times -\colon\catC\to\catC\) has a right adjoint. An object \(A\) of a category \(\catC\) is said to be \emph{exponentiable} if \(\catC\) has products along \(A\) and the functor \(A\times-\) has a right adjoint, usually denoted by \((-)^A\colon\catC\to\catC\), with \(B^A\) said to be an \emph{exponential}, for each object \(B\). It is well-known that both \(\Top\) and \(\Comma{\Top}{X}\) (for every non-empty topological space \(X\)) are not cartesian closed categories, and that the forgetful functor \(\Comma{\Top}{X}\to\Top\) does not preserve exponentials in general. Their exponentiable objects were characterized in \cite{DK70, Nie82, Nie82a, Ric04a, Ric04, CHT03a, RT01}. It is also well-known that the forgetful functor \(\Comma{\Top}{X}\to\Top\) preserves neither products nor exponentials (unless \(X\) is a singleton).

Exponentiability in \(\TopX\), for an Alexandroff space \(X\), is extensively
studied in \cite{Nie06}. Among other results, \cite[Theorem~6.4]{Nie06}
provides a characterisation of exponentiable objects in \(\TopX\), for
particular Alexandroff spaces \(X\). Below we state this result adapted to our
notation, in particular it is formulated with respect to the natural order of
a topological space.

\begin{theorem}[{\cite[Theorem~6.4]{Nie06}}]
  \label{d:thm:6}
  Assume that \(X\) is a complete lattice equipped with the lower Alexandroff topology and assume that the lower Alexandroff topology coincides with the lower topology on \(X\). Then \((A,\alpha)\) is exponentiable in  \(\TopX\) if and only if
  \begin{enumerate}
  \item \(A\) is exponentiable in \(\Top\) and
  \item \(\alpha(a)\wedge- \colon X\to X\) preserves suprema, for all \(a\in X\).
  \end{enumerate}
\end{theorem}

In this section we go substantially beyond the result above and characterise, for suitable (not necessarily Alexandroff) spaces \(X\), the exponentiable objects of \(\TopX\).

By Corollary~\ref{d:cor:1} we know that the existence of finite limits in \(\TopX\) and their preservation by the forgetful functor \(\TopX\to\Top\) is guaranteed exactly when \(X\) is a topological $\wedge$-semilattice. It is therefore not surprising that (see \cite[Lemma~5.1]{Nie06}):
\begin{proposition}\label{prop:Nie}
  If \(X\) is a topological \(\wedge\)-semilattice with bottom element, and \((A,\alpha)\) is an exponentiable object of \(\TopX\), then:
 \begin{enumerate}
 \item The topological space \(A\) is exponentiable in \(\Top\).
 \item The canonical forgetful functor \(\TopX\to\Top\) preserves exponentials, that is, the diagram
   \begin{displaymath}
     \begin{tikzcd}[column sep=large]
       \TopX %
       \ar{r}{(-)^{(A,\alpha)}} %
       \ar{d}[swap]{} %
       & \TopX %
       \ar{d}{} \\
       \Top %
       \ar{r}[swap]{(-)^{A}} %
       & \Top %
     \end{tikzcd}
   \end{displaymath}
   commutes up to isomorphism.
  \end{enumerate}
\end{proposition}
\begin{proof}
  Let \(G \colon \TopX\to\Top\) denote the canonical forgetful functor. Since \(X\) has a top and a bottom element, \(G\) is both a left and a right adjoint. Let \(L\colon\Top\to\TopX\) denote the left adjoint of \(G\), which sends \(f \colon B\to C\) to \(f \colon(B,\bot)\to(C,\bot)\). Regarding the first assertion, just observe that the functor \(A\times- \colon\Top\to\Top\) is the composite \(G\cdot((A,\alpha)\times-)\cdot L\) of left adjoints; regarding the second assertion, just observe that the diagram
  \begin{displaymath}
    \begin{tikzcd}[column sep=large] \Top %
        \ar{r}{A\times-} %
        \ar{d}[swap]{L} %
        & \Top %
        \ar{d}{L} \\
        \TopX %
        \ar{r}[swap]{(A,\alpha)\times-} %
        & \TopX %
      \end{tikzcd}
  \end{displaymath}
  of the corresponding left adjoints commutes up to isomorphism.
\end{proof}

When \(X\) is a topological \(\bigwedge\)-semilattice, the category \(\TopX\)
satisfies the conditions of the Special Adjoint Functor Theorem (see
\cite{Mac98}), that is, \(\TopX\) is cocomplete, well-copowered, with
\((1,\bot)\) being a generator. Hence the functor
\((A,\alpha)\times-\colon\TopX\to\TopX\) has a right adjoint if and only if it
preserves coequalisers and sums. Since \(\TopX\) is infinitely distributive
(see Corollary~\ref{d:cor:3}), the existence of a right adjoint to
\((A,\alpha)\times -\) is equivalent to the preservation of coequalisers.

We can therefore obtain the following criterion for exponentiability of objects in \(\TopX\).

\begin{theorem}
  \label{d:thm:2}
  Assume that \(X\) is a topological \(\bigwedge\)-semilattice. The following
  assertions are equivalent, for an object \((A,\alpha)\) in \(\TopX\).
  \begin{tfae}
    \item
      \label{d:item:12}
      \((A,\alpha)\) is exponentiable in \(\TopX\).
    \item
      \label{d:item:13}
      The topological space \(A\) is exponentiable in \(\Top\) and, for every
      quotient map \(q \colon C\to Q\) in \(\Top\) and every continuous map
      \(\gamma \colon C\to X\),
      \begin{displaymath}
        \alpha \sqcap \Lan_q\gamma=\Lan_{1_A\times q}(\alpha\sqcap \gamma).
      \end{displaymath}
  \end{tfae}
\end{theorem}

\begin{proof}
  \ref{d:item:12}\(\implies\)\ref{d:item:13}: Given such \(q\colon C\to Q\) and \(\gamma\colon C\to X\), and continuous maps \(f,g\colon B\to C\) such that \(q\) is their coequaliser in \(\Top\), we build a coequaliser diagram in \(\TopX\)
  \begin{displaymath}
    \begin{tikzcd}[column sep=large,row sep=large]
      B \ar[shift left=3pt]{r}{f} %
      \ar[shift right=3pt]{r}[swap]{g} %
      \ar[rd,phantom,"\leq",pos=0.6,shift left=2.3mm] %
      \ar[dr,swap,"\beta",bend right=25] %
      & C \ar[d,"\gamma" description] \ar[r,"q"] %
      & Q \ar[dl,phantom,"\leq",pos=0.6,shift right=2.3mm] %
      \ar[dl,dotted,bend left=25,"\Lan_q(\gamma)"]\\
      & X
    \end{tikzcd}
  \end{displaymath}
 with \(\beta=\gamma f\wedge\gamma g\). Exponentiability of \((A,\alpha)\) guarantees that its image under \((A,\alpha)\times-\)
  \begin{displaymath}
    \begin{tikzcd}[column sep=large,row sep=large]
      A\times B \ar[shift left=3pt]{r}{1_A\times f} %
      \ar[shift right=3pt]{r}[swap]{1_A\times g} %
      \ar[rd,phantom,"\leq",pos=0.6,shift left=2.3mm] %
      \ar[dr,swap,"\alpha\sqcap\beta",bend right=25] %
      & A\times C \ar[d,"\alpha\sqcap\gamma" description] %
      \ar[r,"1_A\times q"] %
      & A\times Q \ar[dl,phantom,"\leq",pos=0.6,shift right=2.3mm] %
      \ar[dl,dotted,bend left=25,"\alpha\sqcap\Lan_q(\gamma)"]\\
      & X
    \end{tikzcd}
  \end{displaymath}
 is a coequaliser diagram in \(\TopX\); that is
  \begin{equation}\label{eq:Lan}
    \alpha\sqcap\Lan_q\gamma=\Lan_{1_A\times q}(\alpha\sqcap\gamma).
  \end{equation}
  Conversely, assertion \ref{d:item:13} guarantees that the functor \((A,\alpha)\times-\) preserves coequalisers, and so it is a left adjoint.
\end{proof}

Remark~\ref{d:rem:1} gives us a simple formula to compute the left Kan extension \(\Lan_q\gamma\) when \(X\) is an injective topological space, with \(\gamma\colon C\to X\) and \(q\colon C\to Q\) as above: for every \(y\in Q\),
\[
  \Lan_q\gamma(y)=\bigwedge_{V\in\mathcal{O}(y)}\bigvee_{c\in q^{-1}(V)}\gamma(c).
\]
Using this formula we can replace equation \eqref{eq:Lan} by the more handy equation~\eqref{d:eq:3} below.

\begin{proposition}
  \label{prop:near-distr}
  Assume that \(X\) is an injective topological space and let \((A,\alpha)\) be an object of \(\TopX\). For every quotient map \(q \colon C\to Q\) in \(\Top\) and every continuous map \(\gamma \colon C\to X\),
  \begin{displaymath}
    \alpha\sqcap \Lan_q\gamma=\Lan_{1_A\times q}(\alpha\sqcap \gamma)
  \end{displaymath}
  if and only if, for every \(a\in A\) and \(y\in Q\),
  \begin{equation}
    \label{d:eq:3}
    \bigwedge_{V\in\mathcal{O}(y)}
    \Big(\alpha(a)\wedge\bigvee_{c\in q^{-1}(V)}\gamma(c)\Big)
    =\bigwedge_{V\in\mathcal{O}(y)}
    \bigvee_{c\in q^{-1}(V)}\alpha(a)\wedge\gamma(c).
  \end{equation}
\end{proposition}
\begin{proof}
  By definition, for every \(a\in A\) and \(y\in Q\),
  \begin{align*}
    \Lan_{1_A\times q}(\alpha\sqcap \gamma)(a,y)
    &=\bigwedge_{W\times V\in\mathcal{O}(a,y)}
      \bigvee\{\alpha(a')\wedge\gamma(c)\mid (a',c)\in(1_{A}\times q)^{-1}(W\times V)\}\\
    &=\bigwedge_{V\in\mathcal{O}(y)}\bigwedge_{W\in\mathcal{O}(a)}
      \bigvee_{a'\in W}\bigvee_{c\in q^{-1}(V)}\alpha(a')\wedge\gamma(c).
  \end{align*}
  Using Remark~\ref{d:rem:1},
  \[
    \bigwedge_{W\in\mathcal{O}(a)} \bigvee_{a'\in W}\Big(\bigvee_{c\in
      q^{-1}(V)}\alpha(a')\wedge\gamma(c)\Big)= \bigvee_{c\in
      q^{-1}(V)}\alpha(a)\wedge\gamma(c).
  \]
  Therefore
  \[
    \Lan_{1_A\times
      q}(\alpha\sqcap \gamma)(a,y)=\bigwedge_{V\in\mathcal{O}(y)}\bigvee_{c\in
      q^{-1}(V)}\alpha(a)\wedge\gamma(c).
  \]
  Observing that
  \begin{displaymath}
    (\alpha\sqcap \Lan_q\gamma)(a,y)
    =\alpha(a)\wedge
    \Big(\bigwedge_{V\in\mathcal{O}(y)}\bigvee_{c\in q^{-1}(V)}\gamma(c)\Big)
    =\bigwedge_{V\in\mathcal{O}(y)}
    \Big(\alpha(a)\wedge\bigvee_{c\in q^{-1}(V)}\gamma(c)\Big),
  \end{displaymath}
  we obtain the claimed assertion.
\end{proof}

\begin{theorem}
  \label{d:cor:2}
  Assume that \(X\) is an injective space. Then \((A,\alpha)\) is exponentiable in \(\TopX\) if and only if
  \begin{enumerate}
  \item \(A\) is exponentiable in \(\Top\) and,
    \item\label{d:item:8} for every \(a\in A\), the map \(\alpha(a)\wedge- \colon X\to X\) preserves suprema with respect to the natural order of \(X\).
  \end{enumerate}
\end{theorem}
\begin{proof}
  The ``if-part'' follows from Proposition~\ref{prop:near-distr}. Assume now that \((A,\alpha)\) is exponentiable in \(\TopX\). By Proposition~\ref{prop:Nie}, \(A\) is exponentiable in \(\Top\). To see the second assertion, for a family \(\gamma\colon I\to X\) in \(X\), consider in Proposition~\ref{prop:near-distr} the quotient \(q \colon I\to 1\) with \(I\) being a discrete space.
\end{proof}

\begin{remark}
  Theorem~\ref{d:thm:6} is a special case of our Theorem~\ref{d:cor:2} since the assumption ``the lower topology and the lower Alexandroff topology coincide'' implies that the complete lattice \(X\) is actually an op-continuous lattice. In fact, under this assumption, the lower Scott topology and the lower Alexandroff topology coincide, hence every up-closed subset of \(X\) is also Scott-closed. Therefore, for every codirected subset \(C\) of \(X\), \(\bigwedge C\in{\uparrow}C\), which tells us that this infimum is actually a minimum. As a consequence, \(x\gg x\) for every \(x\in X\). This proves that \(X\) is an op-continuous lattice, even more, \(X\) is an op-algebraic lattice. Also note that every sequence
  \begin{displaymath}
    x_0\geq x_1\geq x_2\geq\dots
  \end{displaymath}
  in \(X\) is stationary.
\end{remark}

Note that, for an injective space \(X\) and \(x\in X\), the map \(x\wedge- \colon X\to X\) is continuous. Moreover, \(x\wedge-\) preserves suprema if and only if \(x\wedge-\) has a right adjoint \(x\Rightarrow- \colon X\to X\) in \(\Ord\). Since \(x\Rightarrow-\), being right adjoint, preserves in particular codirected infima, it is even a continuous map \(x\Rightarrow- \colon X\to X\) (see \cite{Sco72}). Therefore we can equivalently substitute condition \ref{d:item:8} in Theorem~\ref{d:cor:2} by
\begin{enumerate}
\item[\ref{d:item:8}'] for every \(a\in A\), the continuous map \(\alpha(a)\wedge- \colon X\to X\) has a right adjoint in \(\Top\).
\end{enumerate}
In fact, this condition guarantees exponentiability under milder conditions on
\(X\).

\begin{theorem}\label{th:multimap}
  Assume that \(X\) is a topological \(\wedge\)-semilattice. An object \( (A, \alpha) \) of \( \TopX \) is exponentiable provided that
  \begin{enumerate}
  \item \(A\) is exponentiable in \(\Top\),
  \item \(X\) has continuous infima for all families indexed by the underlying set of \(A\), and
  \item for every \(a\in A\), the continuous map \(\alpha(a)\wedge- \colon X\to X\) has a right adjoint \(\alpha(a)\Rightarrow- \colon X\to X\) in \(\Top\).
  \end{enumerate}
\end{theorem}
\begin{proof}
  For \((B,\beta)\) in \(\TopX\), put \((B,\beta)^{(A,\alpha)}=(B^A,\alpha \Rightarrow \beta) \) where $B^A$ is the exponential in \(\Top\), that is, \(B^A=\{h\colon A\to B\mid h\) is continuous\(\}\), and
  \begin{align*}
    \alpha\Rightarrow\beta \colon B^A &\longrightarrow X \\
    h &\longmapsto\bigwedge_{a\in A}\big(\alpha(a)\Rightarrow\beta(h(a))\big).
  \end{align*}
  This map is continuous because it is the composite of continuous maps as indicated in the following diagram
  \begin{displaymath}
    B^A\xrightarrow{\quad\id\quad}B^{|A|}
    \xrightarrow{\quad\beta^{|A|}\quad}X^{|A|}
    \xrightarrow{\quad\widehat{\alpha}\quad}X^{|A|}
    \xrightarrow{\quad\bigwedge\quad}X
  \end{displaymath}
  where $|A|$ denotes the discrete space (with the same underlying set as $A$), and the map $\widehat{\alpha}$, defined by $\widehat{\alpha}((x_a)_{a\in A})=(\alpha(a)\Rightarrow x_a)_{a\in A}$, is continuous because it is induced in the product $X^{|A|}$ by the family of continuous maps
  \begin{displaymath}
    (X^{|A|}\xrightarrow{\quad p_{a}\quad}X
    \xrightarrow{\quad \alpha(a)\Rightarrow-\quad}X)_{a\in A}.
  \end{displaymath}
  It remains to be shown that the evaluation (continuous) map $\ev\colon(A,\alpha)\times(B^A,\alpha\Rightarrow\beta)\to(B,\beta)$ is a morphism in $\TopX$ and has the required universal property. For each $(a,h)\in A\times B^A$, $\alpha(a)\wedge (\alpha\Rightarrow\beta)(h)\leq\beta(h(a))$ by definition of $\alpha\Rightarrow\beta$, hence $\ev$ is in fact a morphism in $\TopX$. Moreover, given a morphism $g\colon(A,\alpha)\times(C,\gamma)\to(B,\beta)$, there exists a unique continuous map $\overline{g}\colon C\to B^A$ so that $\ev\cdot(1_A\times\overline{g})=g$; we only need to show that such map is a morphism $\overline{g}\colon(C,\gamma)\to (B^A,\alpha\Rightarrow\beta)$: for each $a\in A$ and $c\in C$, $\alpha(a)\wedge\gamma(c)\leq\beta(g(a,c))$, and therefore $\gamma(c)\leq\bigwedge_{a\in A}(\alpha(a)\Rightarrow \beta(g(a,c)))=(\alpha\Rightarrow\beta)(\overline{g}(c))$.
\end{proof}

Let us call a topological \(\wedge\)-semilattice \(X\) a \emph{topological Heyting \(\wedge\)-semilattice} if, for every \(x\in X\), the continuous map \(x\wedge-\colon X\to X\) has a right adjoint in \(\Top\). Theorem~\ref{th:multimap}, together with Proposition~\ref{prop:Nie}, gives:

\begin{corollary}
  \label{d:cor:4}
  If \(X\) is a topological Heyting \(\bigwedge\)-semilattice, then the following conditions are equivalent, for an object \((A,\alpha)\) of \(\TopX\):
  \begin{tfae}
  \item \((A,\alpha)\) is exponentiable in \(\TopX\).
  \item \(A\) is exponentiable in \(\Top\).
  \end{tfae}
\end{corollary}

For a topological \(\wedge\)-semilattice \(X\) with bottom element, we analyse now the relationship between the existence of particular exponentials in \(\TopX\) and properties of \(X\). This will be the case of the objects \((1,x)\), where we represent by \(x\) the map \(1\to X\) assigning \(x\in X\) to the unique element of \(1\).

First recall that, by Proposition~\ref{prop:Nie}, for objects \((A,\alpha)\) and \((B,\beta)\) in \(\TopX\) with \((A,\alpha)\) exponentiable, the space \(A\) is exponentiable in \(\Top\) and the exponential \((B,\beta)^{(A,\alpha)}\) in \(\TopX\) is of the form \((B^A,\delta)\), with \(B^A\) denoting the exponential in \(\Top\). Furthermore, for the counit \(\varepsilon_{(B,\beta)}\colon(A,\alpha)\times(B^A ,\delta)\to(B,\beta)\) of the adjunction \((A,\alpha)\times-\dashv(-)^{(A,\alpha)})\) at \((B,\beta)\), the continuous map \(\varepsilon_{(B,\beta)}\colon A\times B^A\to B\) is a counit for the adjunction \(A\times -\dashv (-)^A\) in \(\Top\) at the space \(B\). In fact, one easily verifies the required universal property using that, for every topological space \(C\),
\begin{displaymath}
  \Top(A\times C,B)=(\TopX)((A,\alpha)\times(C,\bot),(B,\beta))
\end{displaymath}
and
\begin{displaymath}
  \Top(C,B^A)=(\TopX)((C,\bot),(B^A,\delta)).
\end{displaymath}
We conclude that the counit \(\varepsilon_{(B,\beta)}\colon(A,\alpha)\times(B^A,\delta)\to(B,\beta)\) can be chosen as \(\varepsilon_{(B,\beta)}(a,h)=h(a)\), for all \(a\in A\) and \(h\in B^A\); that is, it is -- as in \(\Top\) -- the evaluation map. Given also \((C,\gamma)\) in \(\TopX\), it follows immediately that the diagram
\begin{displaymath}
  \begin{tikzcd}
    (\TopX)((C,\gamma),(B,\beta)^{(A,\alpha)}) \ar{r}\ar{d}
    & (\TopX)((A,\alpha)\times (C,\gamma),(B,\beta))
    \ar{d}\\
    \Top(C,B^A) \ar{r}
    & \Top(A\times C,B)
  \end{tikzcd}
\end{displaymath}
commutes. Hence, for a continuous map \(f \colon C\to B^A\) and its mate \(\overline{f}\colon A\times C\to B\) (that is, the continuous map induced by the adjunction), \(f\) is a morphism \((C,\gamma)\to(B,\beta)^{(A,\alpha)}\) in \(\TopX\) if and only if \(\overline{f}\) is a morphism \((A,\alpha)\times(C,\gamma)\to(B,\beta)\) in \(\TopX\).

\begin{theorem}\label{d:thm:5}
  Assume that \(X\) is a topological \(\wedge\)-semilattice with bottom element.
  \begin{enumerate}
  \item\label{d:item:10} For every \(x\in X\), the continuous map \(x\wedge- \colon X\to X\) has a right adjoint in \(\Top\) if and only if \((1,x)\) is exponentiable in \(\TopX\).
  \item\label{d:item:11} For every set \(I\), \(X\) has continuous infima indexed by \(I\) if and only if \((I,\top)\) (with \(I\) considered as a discrete space) is exponentiable in \(\TopX\).
  \end{enumerate}
\end{theorem}
\begin{proof}
  To prove \ref{d:item:10}, consider \(x\in X\). If \(x\wedge- \colon X\to X\) has a right adjoint in \(\Top\), then \((1,x)\) is exponentiable in \(\TopX\) by Theorem~\ref{th:multimap} since every one-element space is exponentiable in \(\Top\). Assume now that \((1,x)\) is exponentiable and consider the exponential \((X,1_X)^{(1,x)}=(X,\gamma_x)\) in \(\TopX\). Then, for all \(y\in X\),
  \begin{displaymath}
    (\TopX)((1,y),(X,1_X)^{(1,x)})\simeq
    (\TopX)((1,x)\times(1,y),(X,1_X)),
  \end{displaymath}
  which is equivalent to
  \begin{displaymath}
    \forall z\in X\,.\,(y\leq\gamma_x(z)\iff x\wedge y\leq z).
  \end{displaymath}
  Therefore \(x\wedge-\dashv\gamma_x\), that is, \(\gamma_x=(x\Rightarrow-)\).

  Regarding \ref{d:item:11}, if \(X\) has continuous infima indexed by \(I\), then \((I,\top)\) is exponentiable by Theorem~\ref{th:multimap} since every discrete space is exponentiable in \(\Top\). On the other hand, note that, for every \((A,\alpha)\) in \(\TopX\), a map \(f\colon I\times A\to X\) is a morphism \(f\colon(I,\top)\times(A,\alpha)\to(X,1_X)\) in \(\TopX\) if and only if, for all \(i\in I\), \(f(i,-)\colon(A,\alpha)\to(X,1_X)\) is a morphism in \(\TopX\). Therefore \((X,1_X)^{(I,\top)}\) coincides with the power \((X,1_X)^I\) in \(\TopX\), and the assertion follows from Proposition~\ref{d:prop:2}.
\end{proof}

Combining Corollary~\ref{d:cor:4} and Theorem~\ref{d:thm:5}, we obtain:

\begin{corollary}
  Assume that \(X\) is a topological \(\wedge\)-semilattice with bottom element. Then the following assertions are equivalent.
  \begin{tfae}
  \item The canonical forgetful functor \(\TopX\to\Top\) reflects exponentiable objects.
  \item \(X\) is a topological Heyting \(\bigwedge\)-semilattice.
  \end{tfae}
\end{corollary}

\begin{examples}
  Assume that the natural order on \(X\) makes it a complete lattice, and that \(X\) is equipped with the lower topology. For each \(x\in X\), if \(x\wedge -\colon X\to X\) has a right adjoint \(x\Rightarrow -\) in \(\Ord\), then, by Lemma~\ref{m:adjoint}, \(x\Rightarrow -\colon X\to X\) is a continuous map. Therefore, if the natural order on \(X\) makes it a frame, then \(X\) is a topological Heyting \(\bigwedge\)-semilattice, and consequently the functor \(\TopX\to\Top\) preserves and reflects exponentiability and exponentials. This is the case when \(X=\bkS,\, \bkN,\, [0,1],\, \bkZ_-^\infty,\,\bkZ_+^\infty\) of Examples \ref{mm:examples1}, but also for \(X=VZ\) the lower Vietoris space of a topological space \(Z\).
\end{examples}

\section{Descent}
\label{sec:descent}

Let \( p \colon A \to B \) be a morphism in a category \( \catC \) with
pullbacks. In our fundamental setting, descent theory for \( p \) refers to
the problem of studying whether bundles over \( A \) equipped with some
additional algebraic structure (specified by \( p \)) -- called
\textit{descent data for p} -- can be regarded as bundles over \( B \).

More specifically, bundles over \( A \) equipped with such descent data for \(
p \) form a category, denoted \( \mathsf{Desc}_{\catC}(p) \), via which the
pullback change-of-base functor \( p^* \colon \Comma{\catC}{B} \to
\Comma{\catC}{A} \) factors through, depicted as follows:
\begin{equation}
  \label{eq:pb.changebase}
  \begin{tikzcd}
    \Comma{\catC}{B} \ar[rd,swap,"\mathcal K^p"]
                     \ar[rr,"p^*"]
    && \Comma{\catC}{A}  \\
    & \mathsf{Desc}_{\catC}(p) \ar[ru,swap,"\mathcal U^p"]
  \end{tikzcd}
\end{equation}
In Diagram \eqref{eq:pb.changebase}, the functor \( \mathcal U^p \) forgets
the descent data, and \( \mathcal K^p \) -- \textit{the comparison functor} --
maps a bundle \( f \colon U \to B \) to the bundle \( p^*(f) \colon p^*(U)
\to A \) equipped with its canonical descent data.

By the Bénabou-Roubaud theorem \cite{BR70} (see \cite{JT94, JST04, Nun22} for a modern
account), Diagram \eqref{eq:pb.changebase} is equivalent to the
Eilenberg-Moore factorization of \( p^* \) with respect to the adjunction \(
p_! \dashv p^* \), so that the category \( \mathsf{Desc}_{\catC}(p) \) can be
regarded as the category of algebras for the induced monad.

We say \( p \) is an \textit{effective descent morphism} if \( \mathcal K^p \)
is an equivalence, that is, if \( p^* \) is monadic. Similarly, \(p^*\) is a
\emph{descent morphism} if \( \mathcal K^p \) is fully faithful, that is, if
\( p^* \) is premonadic. If \(\catC\) is finitely complete, then a morphism in
\(\catC\) is a descent morphism if and only if it is a pullback-stable regular
epimorphism \cite{JT94, JST04}.

The descent morphisms in \(\Top\) are precisely the \emph{biquotient} maps
(also known as \emph{limit lifting maps} and as \emph{universal quotient
maps}). They are those continuous maps \(f \colon A\to B\) such that, for
every ultrafilter convergence \(\mathfrak{b}\leadsto b\) in \(B\), there
exists an ultrafilter convergence \(\mathfrak{a}\leadsto a\) in \(A\) with
\(Uf(\mathfrak{a})=\mathfrak{b}\) and \(f(a)=b\) (see \cite{Haj66, Haj67,
Mic68, DK70}). Effective descent morphisms in \(\Top\) were first
characterised by \cite{RT94} via a lifting property on ``2-chains of
ultrafilter convergence''. For more information we refer to \cite{CH02, CJ11};
however, in this paper we will not make explicit use of this characterisation.

In this section we give, for suitable spaces \(X\), a characterisation of
effective descent morphisms in \(\TopX\).

In what follows, we study the effective descent morphisms of \( \TopX \), for
suitable spaces \( X \).

\begin{center}
  \emph{Throughout this section we assume that \(X\) is a topological
  \(\bigwedge\)-semilattice}.
\end{center}

We provide sufficient conditions for a morphism \( p \colon (A,\alpha) \to
(B,\beta) \) in \( \TopX \) to be effective for descent in
Theorem~\ref{d:thm:7}, and in Theorem~\ref{d:thm:8} we confirm that these
conditions are necessary when \( X \) is a frame with respect to its natural
order.

We leave the complete characterisation of the effective descent morphisms in
\( \TopX \) as an open problem. The obstacle in our approach lies in the
preservation of effective descent morphisms by the functor \( \TopX \to
\Fam(X) \), which we analyse in Lemmas~\ref{lem:preserve-term-cod-suff}
and~\ref{d:lem:1}. This analysis is based on more general preservation
techniques which are developed in \cite{NP24_tmp}.

We start with the following observation.
\begin{lemma}
  \label{d:lem:3}
  The canonical forgetful functor \(\TopX\to\Top\) preserves (effective)
  descent morphisms.
\end{lemma}
\begin{proof}
  As in \cite[Theorem~3.3]{CN23}.
\end{proof}

Besides \(\Top\), we shall also compare \(\TopX\) with the category
\(\Fam(X)\) of \emph{families in \(X\)} defined as follows: the objects of
\(\Fam(X)\) are families \((x_i)_{i\in I}\) of elements of \(X\) indexed by a
set \(I\), and a morphism \(f \colon(x_i)_{i\in I}\to(y_j)_{j\in J}\) in
\(\Fam(X)\) is given by a map \(f \colon I\to J\) satisfying \(x_i\leq
y_{f(i)}\), for all \(i\in I\). Identities and composition are inherited from
\(\Set\).
\begin{lemma}
  \label{d:lem:4}
  The canonical forgetful functor \(\TopX\to\Fam(X)\) preserves:
  \begin{enumerate}
    \item
      \label{r:item:pullbacks}
      pullbacks;
    \item
      \label{d:item:22}
      regular epimorphisms of type \((B,\beta)\to(1,\gamma)\);
    \item
      \label{d:item:23}
      descent morphisms.
  \end{enumerate}
\end{lemma}
\begin{proof}
  We confirm \ref{r:item:pullbacks} by noting that the functor \( \TopX \to
  \Fam(X) \) has a left adjoint, mapping each family \( (\alpha(a))_{a \in A}
  \) to the pair \( (A,\alpha) \) where \( A \) has the discrete topology: we
  simply note that, if \( (B,\beta) \) is an object in \( \TopX \), morphisms
  \( (A,\alpha) \to (B,\beta) \) in \( \TopX \) are in a natural bijective
  correspondence with morphisms \( (\alpha(a))_{a \in A} \to (\beta(b))_{b \in
  B} \).

  Regarding \ref{d:item:22}, observe that \(q \colon(B,\beta)\to(1,\gamma)\) is a regular epimorphism in \(\TopX\) if and only if \(B\neq\varnothing\) and
  \begin{displaymath}
    \gamma=\bigvee\{\beta(b)\mid b\in B\}.
  \end{displaymath}
  To see \ref{d:item:23}, let now \(q \colon(B,\beta)\to(C,\gamma)\) be a pullback-stable regular epimorphism in \(\TopX\). Then, for every \(c\in C\),
  \begin{displaymath}
    \gamma(c)=\bigvee\{\beta(b)\mid b\in B,\,f(b)=c\}
  \end{displaymath}
  since in the pullback diagram
  \begin{displaymath}
    \begin{tikzcd}
      (B_{c},\beta) %
      \ar{r}{q} %
      \ar{d}[swap]{} %
      & (1,\gamma(c)) %
      \ar{d}{c} \\
      (B,\beta) %
      \ar{r}[swap]{q} %
      & (C,\gamma) %
    \end{tikzcd}
  \end{displaymath}
  the top row is a regular epimorphism in \(\TopX\). We conclude that \(q
  \colon(\beta(b))_{b\in B}\to(\gamma(c))_{c\in C}\) is a regular epimorphism
  in \(\Fam(X)\). It is indeed a pullback-stable regular epimorphism in
  \(\Fam(X)\) because for every pullback diagram
  \begin{displaymath}
    \begin{tikzcd}
      (\delta(d))_{d\in D} %
      \ar{r}{\pi_{2}} %
      \ar{d}[swap]{\pi_{1}} %
      & (\alpha(a))_{a\in A} %
      \ar{d}{g} \\
      (\beta(b))_{b\in B} %
      \ar{r}[swap]{q} %
      & (\gamma(c))_{c\in C} %
    \end{tikzcd}
  \end{displaymath}
  in \(\Fam(X)\) we may consider the discrete topology on \(A\) and the pullback diagram
  \begin{displaymath}
    \begin{tikzcd} 
      (D,\delta) %
      \ar{r}{\pi_{2}} %
      \ar{d}[swap]{\pi_{1}} %
      & (A,\alpha) %
      \ar{d}{g} \\
      (B,\beta) %
      \ar{r}[swap]{q} %
      & (C, \gamma) %
    \end{tikzcd}
  \end{displaymath}
  in \(\TopX\). Then \(\pi_2 \colon (D, \delta)\to(A,\alpha)\) is a
  pullback-stable regular epimorphism in \(\TopX\) and therefore  \(\pi_2
  \colon (\delta(d))_{d\in D}\to(\alpha(a))_{a\in A}\) is a regular
  epimorphism in \(\Fam(X)\).
\end{proof}

\begin{lemma}
  \label{lem:preserve-term-cod-suff}
  Let \( q \colon (B,\beta) \to (C,\gamma) \) be an effective descent morphism
  in \( \TopX \), and consider the following pullback diagram for each \(c \in
  C \):
  \begin{equation}
    \label{eq:pb-point}
    \begin{tikzcd}
      (\beta^{-1}(c),\beta|_{\beta^{-1}(c)})
        \ar[r,"q|_{\beta^{-1}(c)}"]
        \ar[d]
        & (1,\gamma(c)) \ar[d,"c"] \\
      (B,\beta) \ar[r,swap,"q"]
        & (C,\gamma)
    \end{tikzcd}
  \end{equation}
  We have that \( q|_{\beta^{-1}(c)} \) is an effective descent morphism in \(
  \Fam(X) \) for all \( c \) if and only if \( q \) is an effective descent
  morphism in \( \Fam(X) \).

  \end{lemma}

\begin{proof}
  We note that, in an extensive category, coproducts of effective
  descent morphisms are effective for descent, and the underlying morphism \(
  q \colon (B,\beta) \to (C,\gamma) \) in \( \Fam(X) \) is the coproduct of
  the family of effective descent morphisms \( (q|_{\beta^{-1}(c)})_{c \in C} \).

  The converse holds by pullback stability of effective descent morphisms.
\end{proof}

Lemma~\ref{lem:preserve-term-cod-suff} says that the canonical forgetful
functor \( \TopX \to \Fam(X) \) preserves the effective descent morphisms \(
q|_{\beta^{-1}(c)} \) for all \( c \) if and only if it preserves \( q \).
Hence, in order to study the preservation of effective descent morphisms by \(
\TopX \to \Fam(X) \), it is sufficient to analyse whether it preserves
effective descent morphisms of the form \( (B,\beta) \to (1,x) \).

\begin{lemma}
  Let \((B,\beta)\) in \(\TopX\) so that the image of \(\beta \colon B\to X\) has a maximum. Then \(\TopX \to\Fam(X)\) preserves all effective descent morphisms \((B,\beta)\to(1,x)\).
\end{lemma}
\begin{proof}
When \(\beta \colon B\to X\) has a maximum, the morphism \((B,\beta)\to(1,x)\) is a split epimorphism, hence it is immediately an effective descent morphism preserved by any functor.
\end{proof}

The following lemma gives a general description of the effective descent morphisms in \(\TopX\) which are preserved by \( \TopX \to \Fam(X) \).

\begin{lemma}
  \label{d:lem:1}
  Let \( (B,\beta) \to (1,x) \) be an effective descent morphism in \( \TopX
  \). The following are equivalent:
  \begin{tfae}
    \item
      \label{tfae:edm-famx}
      \( p \colon (B,\beta) \to (1,x) \) is an effective descent morphism in
      \( \Fam(X) \).
    \item
      \label{tfae:cond}
      For all families of points \( (\theta(b))_{b \in B} \) in \( X \)  with
      \( \theta(b) \leq \beta(b) \) for all \( b \) satisfying \( \beta(b')
      \land \theta(b) = \theta(b') \land \beta(b) \) for all pairs \( b,b' \in
      B \), we have \( \theta \colon B \to X \) continuous.
  \end{tfae}
\end{lemma}

\begin{proof}
  First, we note that the following diagram commutes, up to isomorphism:
  \begin{equation}
    \label{eq:em-comp}
    \begin{tikzcd}
      \Comma{\TopX}{(1,x)} \ar[r,"\mathcal K^p"] \ar[d,swap,"H"]
        & \mathsf{Desc}_{\TopX}(p) \ar[d,"\tilde H"] \\
      \Comma{\Fam(X)}{(1,x)} \ar[r,swap,"\mathcal K^p"]
        & \mathsf{Desc}_{\Fam(X)}(p)
    \end{tikzcd}
  \end{equation}
  Moreover, we know that \( H \) is essentially surjective.

  Second, we note that a family of points \( (\theta(b))_{b \in B} \) in \( X
  \)  with \( \theta(b) \leq \beta(b) \) for all \( b \) satisfying \(
  \beta(b') \land \theta(b) = \theta(b') \land \beta(b) \) is precisely a
  \textit{connected} descent datum \( \id \colon (B,\theta) \to (B,\beta) \)
  in \( \Fam(X) \) for \( p \) in the sense of \cite{Pre23,Pre24}.

  Let us assume \ref{tfae:edm-famx} holds, in which case \( \tilde H \) must
  be essentially surjective. We note that descent data for \( p \) in \( \TopX
  \) must be of the form \( \pi_2 \colon (B\times C,\gamma) \to (B, \beta) \),
  where \( \pi_2 \) is the projection \( B\times C \to B \). Thus, given
  descent data \( \id \colon (B,\theta) \to (B,\beta)\) for \( p \) in \(
  \Fam(X) \), the only topology on the underlying set of \( B \) for which \(
  \id \colon (B,\theta) \to (B,\beta) \) still defines descent data for \( p
  \) in \( \TopX \) is the topology on \( B \) itself, meaning \( \theta \)
  must be continuous, thereby confirming \ref{tfae:cond}.

  Conversely, if \ref{tfae:cond} holds -- that is, for any connected descent
  data \( \id \colon (B,\theta) \to (B,\beta) \) for \( p \) in \( \Fam(X) \)
  we have \( \theta \) continuous -- we have that \( \id \colon (B,\theta) \to
  (B,\beta) \) determines descent data for \( p \) in \( \TopX \). As \( p \)
  is effective for descent in \( \TopX \), we conclude that this descent data
  corresponds to the bundle \( (1,\bigvee_{b \in B}\theta(b)) \to (1,x) \),
  whose pullback along \( p \) guarantees that \( \theta(b) = \beta(b) \land
  \bigvee_{b' \in B} \theta(b') \) for all \( b \in B \). This confirms
  that \( p \) is effective for descent in \( \Fam(X) \).
\end{proof}

Following \cite{CP25}, we consider the (pseudo)pullback
\begin{displaymath}
  \begin{tikzcd}
    \Top\times_{\Set}\Fam(X) %
    \ar{r}{\rho_{2}} %
    \ar{d}[swap]{\rho_{1}} %
    & \Fam(X) %
    \ar{d}{} \\
    \Top %
    \ar{r}[swap]{} %
    & \Set %
  \end{tikzcd}
\end{displaymath}
and the canonical functor
\begin{displaymath}
  H \colon\TopX \longrightarrow \Top\times_{\Set}\Fam(X).
\end{displaymath}
induced by the forgetful functors
\begin{displaymath}
  \TopX \longrightarrow\Top
  \quad\text{and}\quad
  \TopX \longrightarrow\Fam(X).
\end{displaymath}

\begin{proposition}
  \label{d:prop:3}
  \begin{enumerate}
  \item\label{d:item:17} A morphism \(f\) in \(\Top\times_{\Set}\Fam(X)\) is effective for descent if and only if \(\rho_2(f)\) is effective for descent in \(\Fam(X)\) and \(\rho_1(f)\) is effective for descent in \(\Top\).
  \item\label{d:item:18} The functor \(\TopX \to \Top\times_{\Set}\Fam(X)\) is fully faithful and preserves pullbacks.
  \end{enumerate}
\end{proposition}
\begin{proof}
  Regarding \ref{d:item:17}, both functors \(\rho_1\) and \(\rho_2\) preserve
  effective descent morphism by \cite[Corollary~2.3]{CJ24}. On
  the other hand, if \(\rho_1(f)\) and \(\rho_2(f)\) are effective for descent
  in \(\Top\) and \(\Fam(X)\), respectively, then either
  \cite[Corollary~2.6]{CJ24} or \cite[Corollary 9.6]{Nun18} guarantee that
  \(f\) is effective for descent in \(\Top\times_{\Set}\Fam(X)\). The
  assertion \ref{d:item:18} is clear.
\end{proof}

\begin{corollary}
  \label{d:cor:6}
  Let \(f \colon(A,\alpha)\to(B,\beta)\) be a morphism in \(\TopX\) so that \(f \colon (\alpha(a))_{a\in A}\to(\beta(b))_{b\in B}\) is effective for descent in \(\Fam(X)\). Then \(f\) is effective for descent in \(\TopX\) if and only if \(f \colon A\to B\) is effective for descent in \(\Top\) and, for every pullback
  \begin{equation}
    \label{d:eq:4}
    \begin{tikzcd}
      H(D,\delta) %
      \ar{r}{\pi_2} %
      \ar{d}[swap]{\pi_1} %
      & (C,(\gamma_c)_{c\in C}) %
      \ar{d}{g} \\
      H(A,\alpha) %
      \ar{r}[swap]{Hf} %
      & H(B,\beta) %
    \end{tikzcd}
  \end{equation}
  in \(\Top\times_{\Set}\Fam(X)\), the family \(\gamma=(\gamma_c)_{c\in C}\) is a continuous map \(\gamma \colon C\to X\).
\end{corollary}
\begin{proof}
  First note that \(f\) being effective for descent in \(\TopX\) implies that \(f\) is effective descent in \(\Top\) (by Lemma~\ref{d:lem:3}) and, by Proposition~\ref{d:prop:3}~\ref{d:item:17}, it is also effective for descent in \(\Top\times_{\Set}\Fam(X)\). The assertion follows now from \cite[Proposition~2.6]{JT94} using Proposition~\ref{d:prop:3}.
\end{proof}

Recall from Section~\ref{sec:topologicity} that \(e_X \colon X\to UX\) has a left adjoint in \(\Ord\) if and only if every ultrafilter in \(X\) has a smallest convergence point with respect to the natural order in \(X\). We say that \emph{binary infima preserve smallest convergence points in \(X\)} whenever \(e_X \colon X\to UX\) has a left adjoint \(\sigma \colon UX\to X\) in \(\Ord\) and the diagram
\begin{displaymath}
  \begin{tikzcd}
    U(X\times X)%
    \ar{r}{U\wedge}%
    \ar{d}[swap]{\mathrm{can}}%
    & UX%
    \ar{dd}{\sigma}\\
    UX\times UX%
    \ar{d}[swap]{\sigma\times\sigma}\\
    X\times X%
    \ar{r}[swap]{\wedge}%
    & X
  \end{tikzcd}
\end{displaymath}
commutes.

\begin{lemma}
  \label{d:lem:2}
  Assume that binary infima preserve smallest convergence points in \(X\). In
  the pullback~\eqref{d:eq:4}, let \(\mathfrak{w}\in UD\), and put
  \(\mathfrak{a}=U\pi_1(\mathfrak{w})\) and
  \(\mathfrak{c}=U\pi_2(\mathfrak{w})\). Then
  \begin{displaymath}
    \sigma(U\delta(\mathfrak{w}))
    =\sigma(U\alpha(\mathfrak{a}))\wedge\sigma(U\gamma(\mathfrak{c})).
  \end{displaymath}
\end{lemma}
\begin{proof}
  Observe that in the diagram
  \begin{displaymath}
    \begin{tikzcd}[column sep=large]
      UD\ar{r}%
      & U(A\times C)%
      \ar{r}{U(\alpha\times\gamma)}\ar{d}%
      & U(X\times X)
      \ar{r}{U{\wedge}}\ar{d}%
      & UX\ar{dd}{\sigma}\\
      & UA\times UC%
      \ar{r}[swap]{U\alpha\times U\gamma}%
      & UX\times UX%
      \ar{d}[swap]{\sigma\times\sigma}\\
      && X\times X%
      \ar{r}[swap]{\wedge} & X
    \end{tikzcd}
  \end{displaymath}
  both squares commute, the top row represents \(U\delta \colon UD\to UX\) and the map corresponding to the bottom path sends \(\mathfrak{w}\in UD\) to \(\sigma(U\alpha(\mathfrak{a}))\wedge\sigma(U\gamma(\mathfrak{c}))\).
\end{proof}

\begin{theorem}
  \label{d:thm:7}
  Assume that binary infima preserve smallest convergence points in the topological \(\bigwedge\)-semilattice \(X\). Let \(f \colon(A,\alpha)\to(B,\beta)\) be a morphism in \(\TopX\) so that \(f \colon (\alpha(a))_{a\in A}\to(\beta(b))_{b\in B}\) is effective for descent in \(\Fam(X)\). Then \(f \colon(A,\alpha)\to(B,\beta)\) is effective for descent in \(\TopX\) if and only if the following two conditions hold.
  \begin{enumerate}
  \item\label{d:item:14} \(f \colon A\to B\) is effective for descent in \(\Top\).
  \item\label{d:item:15} For all ultrafilter convergences \(\mathfrak{b}\leadsto b\) in \(B\) and for all \(w\in X\) with \(w\leq  \sigma(U\beta(\mathfrak{b}))\), we have
    \begin{equation}
      \label{r:eq:distr-cond}
      w = \bigvee_{%
        \substack{%
          \mathfrak a \leadsto a \\
          Uf(\mathfrak{a}) = \mathfrak{b}, \\
          f(a) = b}}
      w \wedge \sigma(U\alpha(\mathfrak{a})).
    \end{equation}
  \end{enumerate}
\end{theorem}
\begin{proof}
  We show first that conditions~\ref{d:item:14} and \ref{d:item:15} are
  sufficient for \(f\) to be effective for descent in \(\TopX\). To do so,
  consider a pullback in \(\Top\times_{\Set}\Fam(X)\) as in \eqref{d:eq:4}.
  We want to show that \( \gamma \) is continuous.  For \(\mathfrak{c}\leadsto
  c\) an ultrafilter convergence in \(C\), we show that
  \(U\gamma(\mathfrak{c})\leadsto\gamma(c)\). Put
  \(\mathfrak{b}=Ug(\mathfrak{c})\) and \(b=g(c)\). Then, since the functor
  \(U \colon\Top\to\Top\) is locally monotone, \(\sigma \colon UX\to X\) is
  monotone and we have \(\sigma(U\gamma(\mathfrak{c})) \leq
  \sigma(U\beta(\mathfrak{b}))\) (in order to apply \( U \) to \( \gamma \) we
  consider \( C \) with the discrete topology).  Therefore, by
  condition~\ref{d:item:15},
  \begin{displaymath}
    \sigma(U\gamma(\mathfrak{c}))
    = \bigvee_{%
      \substack{%
        \mathfrak{a}\leadsto a \\
        Uf(\mathfrak{a}) = \mathfrak{b}, \\
        f(a) = b}}
    \sigma(U\gamma(\mathfrak{c})) \wedge \sigma(U\alpha(\mathfrak{a})).
  \end{displaymath}
  Since the functor \(U \colon\Set\to\Set\) weakly preserves pullbacks, for every \(\mathfrak{a}\in UA\) with \(Uf(\mathfrak{a}) = \mathfrak{b}\), there is some \(\mathfrak{w}\in UD\) with \(U\pi_1(\mathfrak{w})=\mathfrak{a}\) and \(U\pi_2(\mathfrak{w})=\mathfrak{c}\). Therefore, applying also Lemma~\ref{d:lem:2}, we obtain
  \begin{displaymath}
    \bigvee_{%
      \substack{%
        \mathfrak{a}\leadsto a \\
        Uf(\mathfrak{a}) = \mathfrak{b}, \\
        f(a) = b}}
    \sigma(U\gamma(\mathfrak{c})) \wedge \sigma(U\alpha(\mathfrak{a}))
    =
    \bigvee_{%
      \substack{%
        \mathfrak{w}\in UD \\
        U\pi_2(\mathfrak{w}) = \mathfrak{c}, \\
        U\pi_1(\mathfrak{w})\leadsto a, \\
        f(a)=b}}
    \sigma(U\delta(\mathfrak{w})).
  \end{displaymath}
  For every \(\mathfrak{w}\in UD\) with \(U\pi_2(\mathfrak{w}) = \mathfrak{c}\), \(U\pi_1(\mathfrak{w})\leadsto a\) and \(f(a)=b\), we have that \(\mathfrak{w}\leadsto(a,c)\) in \(D\). Since \(\delta \colon D\to X\) is continuous, \(U\delta(\mathfrak{w})\leadsto\delta(a,c)\), which is equivalent to \( \sigma(U\delta(\mathfrak{w}))\leq\delta(a,c)\). Finally, by definition,
  \begin{displaymath}
    \delta(a,c)=\alpha(a)\wedge\gamma(c)\leq\gamma(c).
  \end{displaymath}
  Summing up, \(\sigma(U\gamma(\mathfrak{c}))\leq\gamma(c)\), that is, \(U\gamma(\mathfrak{c})\leadsto\gamma(c)\). We conclude that \(\gamma \colon C\to X\) is continuous. By Corollary~\ref{d:cor:6}, \(f\) is effective for descent in \(\TopX\).

  Conversely, assume now that \(f\colon(A,\alpha)\to(B,\beta)\) is effective for descent in \(\TopX\). By Lemma~\ref{d:lem:3}, \(f \colon A\to B\) is effective for descent in \(\Top\).
  To verify Condition~\ref{d:item:15}, fix an ultrafilter convergence \(\mathfrak{b}\leadsto b\) in \(B\) and \(w\in X\) with \(w\leq\sigma(U\beta(\mathfrak{b}))\). Consider the set \(C=B+\{\star\}\) (for simplicity of notation, we assume \(\star\notin B\) so that \(C=B\cup\{\star\}\)) together with the topology where the only non-trivial ultrafilter convergence is \(\mathfrak{b}\leadsto\star\). Furthermore, we define the maps \(g \colon C\to B\) with \(g(\star)=b\) and \(g(c)=c\) for every \(c\in B\), and
  \begin{align*}
    \gamma \colon C & \longrightarrow X\\
    c & \longmapsto
        \begin{cases}
          \beta(c)\wedge w &\text{if \(c\in B\)}\\
          \displaystyle\bigvee_{%
          \substack{%
          \mathfrak a \leadsto a \\
          Uf(\mathfrak{a}) = \mathfrak{b}, \\
          f(a) = b}}
          w \wedge \sigma(U\alpha(\mathfrak{a}))
                           & \text{if \(c=\star\)}
        \end{cases}
  \end{align*}
  For the pullback~\eqref{d:eq:4}, we show now that \(\delta \colon D\to X\) is continuous. We can write the set \(D\) as
  \begin{displaymath}
    D=\underbrace{\{(a,c)\in D\mid c\in B\}}_{=D_1}
    +\underbrace{\{(a,c)\in D\mid c=\star\}}_{=D_2}.
  \end{displaymath}
  Let \(\mathfrak{d}\leadsto(a,c)\) be an ultrafilter convergence in \(D\). We put \(\mathfrak{a}=U\pi_1(\mathfrak{d})\) and \(\mathfrak{c}=U\pi_1(\mathfrak{d})\).
  \begin{itemize}
  \item Assume that \(\mathfrak{d}\in UD_1\) and \(c\in B\). Then
    \begin{displaymath}
      \sigma(U\delta(\mathfrak{d})) %
      \leq \sigma(U\alpha(\mathfrak{a}))\wedge\sigma(U\gamma(\mathfrak{c})) %
      \leq \sigma(U\alpha(\mathfrak{a}))\wedge w %
      \leq \alpha(a)\wedge w %
      \leq \alpha(a)\wedge\beta(f(a))\wedge w %
      =\delta(a,c).
    \end{displaymath}
  \item Assume that \(\mathfrak{d}\in UD_2\) and \(c=\star\). Then \(\mathfrak{c}=\doo{\star}\) and
    \begin{displaymath}
      \sigma(U\delta(\mathfrak{d})) %
      \leq \sigma(U\alpha(\mathfrak{a}))\wedge\gamma(\star) %
      \leq \alpha(a)\wedge\gamma(\star)=\delta(a,c).
    \end{displaymath}
  \item Assume that \(\mathfrak{d}\in UD_1\) and \(c=\star\). Then \(\mathfrak{c}=\mathfrak{b}\) and therefore \(Uf(\mathfrak{a})=\mathfrak{b}\). Hence, \(\sigma(U\alpha(\mathfrak{a}))\wedge w=\gamma(\star)\). Since also \(\sigma(U\alpha(\mathfrak{a}))\leq \alpha(a)\), we obtain
    \begin{displaymath}
      \alpha(U\delta(\mathfrak{d}))%
      \leq \sigma(U\alpha(\mathfrak{a}))\wedge w %
      \leq \alpha(a)\wedge\gamma(\star)=\delta(a,\star).
    \end{displaymath}
  \item Finally, \(\mathfrak{d}\in UD_2\) and \(c\in B\) is not possible by the definition of the convergence in \(C\).
  \end{itemize}
  Hence, we conclude that \(U\delta(\mathfrak{d})\leadsto\delta(a,c)\) in \(X\). This proves that \(\delta \colon D\to X\) is continuous. By Corollary~\ref{d:cor:6}, \(\gamma \colon C\to X\) is continuous as well. Therefore \(\sigma(U\gamma(\mathfrak{b}))\leq\gamma(\star)\). Since also
  \begin{displaymath}
    \sigma(U\gamma(\mathfrak{b}))
    =\sigma(U\beta(\mathfrak{b}))\wedge w=w,
  \end{displaymath}
  \(w\leq \gamma(\star)\) and we have proven \ref{d:item:15}.
\end{proof}

\begin{lemma}
  Let \(f \colon (A,\alpha)\to(B,\beta)\) be a morphism in \(\TopX\) satisfying Condition~\ref{d:item:15} of Theorem~\ref{d:thm:7}. Then \(f \colon (\alpha(a))_{a\in A}\to(\beta(b))_{b\in B}\) is a descent morphism in \(\Fam(X)\).
\end{lemma}
\begin{proof}
  By \cite[Lemma~2.1]{CP25} (see also \cite{Pre24, Pre23}), the morphism \(f \colon (\alpha(a))_{a\in A}\to(\beta(b))_{b\in B}\) is a descent morphism in \(\Fam(X)\) if and only if, for all \(b\in B\) and \(w\leq\beta(b)\),
  \begin{displaymath}
    w=\bigvee_{a,\,f(a)=b}w\wedge\alpha(a).
  \end{displaymath}
  To verify this condition, consider \((\mathfrak{b}\leadsto b)=(\doo{b}\leadsto b)\) in Condition~\ref{d:item:15} of Theorem~\ref{d:thm:7} and observe that \(\sigma(U\alpha(\mathfrak{a}))\leq\alpha(a)\) for every \(\mathfrak{a}\leadsto a\) in \(A\).
\end{proof}

By \cite[Corollary 4.6]{Pre23}, every descent morphism in \(\Fam(X)\) is
effective for descent provided that \(X\) is also a frame. Therefore we can
state:
\begin{theorem}
  \label{d:thm:8}
  Assume that binary infima preserve smallest convergence points in the topological \(\bigwedge\)-semilattice \(X\) and, moreover, that the ordered set \(X\) is a frame. Then a morphism \(f \colon(A,\alpha)\to(B,\beta)\) in \(\TopX\) is effective for descent if and only if
  \begin{enumerate}
  \item \(f \colon A\to B\) is effective for descent in \(\Top\),
  \item\label{d:item:25} for all ultrafilter convergences \(\mathfrak{b}\leadsto b\) in \(B\), we have
    \begin{displaymath}
      \sigma(U\beta(\mathfrak{b}))
      = \bigvee_{%
        \substack{%
          \mathfrak a \leadsto a \\
          Uf(\mathfrak{a}) = \mathfrak{b}, \\
          f(a) = b}}
       \sigma(U\alpha(\mathfrak{a})).
    \end{displaymath}
  \end{enumerate}
\end{theorem}

There is a natural way of producing spaces \(X\) satisfying the hypotheses of
Theorems \ref{d:thm:7} and \ref{d:thm:8}, as follows.

\begin{proposition}\label{m:prop}
If \(X\) is a continuous lattice equipped with the lower topology, then \(X\) is a topological \(\bigwedge\)-semilattice, with smallest convergence points preserved by binary infima.
\end{proposition}

\begin{proof}
  To prove this statement, we will explain how a continuous lattice equipped with the lower topology can be obtained as the dual space of an injective space. As we have already observed, if \(Y\) is an injective space, then \(Y\) is an op-continuous lattice with respect to the natural order and the topology of \(Y\) coincides with the lower Scott topology. Furthermore, every ultrafilter in \(Y\) has smallest convergence points; in fact, the left adjoint \(\sigma \colon UY\to Y\) of \(e_Y \colon Y\to UY\) is precisely the convergence of the Lawson topology of \(Y\) \cite[Proposition~VII-3.10]{GHK+03}. Explicitly,
  \begin{displaymath}
    \sigma(\mathfrak{y})=\bigwedge_{A\in \mathfrak{y}}\bigvee A,
  \end{displaymath}
  for all \(\mathfrak{y}\in UY\). The space \(Y\) is stably compact, hence we can define its dual space \(Y^{\mathrm{op}}\), whose topology is given by the convergence relation
  \begin{displaymath}
    \mathfrak{y}\leadsto y\iff y\leq \sigma(\mathfrak{y});
  \end{displaymath}
  this is the so-called cocompact topology. Hence the natural order of \(Y^{\mathrm{op}}\) is the dual of the natural order of \(Y\) and also in \(Y^{\mathrm{op}}\) every ultrafilter \(\mathfrak{x}\) has \(\sigma(\mathfrak{x})\) as smallest convergence point.
  Furthermore, by \cite[Proposition~VI-6.24]{GHK+03}, the cocompact topology of \(Y\) is the upper topology with respect to the natural order of \(Y\), hence the lower topology with respect to the natural order of \(Y^{\mathrm{op}}\). By Lemma~\ref{m:adjoint}, \(Y^{\mathrm{op}}\) is a topological \(\bigwedge\)-semilattice. Finally, since the map
  \begin{displaymath}
    \vee \colon Y\times Y \longrightarrow Y
  \end{displaymath}
  is continuous and left adjoint to \(\Delta_Y \colon Y\to Y\times Y\), it preserves smallest convergence points of ultrafilters. We conclude that the binary infimum operation
  \begin{displaymath}
    \wedge \colon Y^{\mathrm{op}}\times Y^{\mathrm{op}} \longrightarrow Y^{\mathrm{op}}
  \end{displaymath}
in \(Y^{\mathrm{op}}\) preserves smallest convergence points of ultrafilters.

 In conclusion,  when \(X\) is a continuous lattice equipped with the lower topology, denoting by \(Y\) the topological space with the same underlying set and whose topology is the lower Scott topology with respect to the reverse of the natural order of \(X\), we obtain \(X=Y^{\mathrm{op}}\), and the result follows.
\end{proof}
We note that, if \(X\) is a completely distributive lattice (see \cite{Ran52, FW90} and \cite[Corollary~I-2.9]{GHK+03}) equipped with the lower topology, in particular if \(X=\bkS\) and \(X=[0,1]\) as in Examples~\ref{mm:examples1}, then \(X\) is both a continuous lattice and a frame. Hence:

\begin{corollary}
If \(X\) is a completely distributive lattice equipped with the lower topology, then \(X\) satisfies the hypotheses of Theorem~\ref{d:thm:8}.
\end{corollary}

\begin{example}
  In \(\LaxComma{\Top}{\bkS}\), by Theorem~\ref{d:thm:8}, \(f
  \colon(A,A_0)\to(B,B_0)\) is an effective descent morphism if and only if
  \(f \colon A\to B\) is an effective descent morphism in \(\Top\) and \(f_0
  \colon A_0\to B_0\) is a descent morphism in \(\Top\). In fact,
  Condition~\ref{d:item:25} of Theorem~\ref{d:thm:8} specialises to: for all
  ultrafilter convergences \(\mathfrak{b}\leadsto b\) in \(B\) with \(B_0\in
  \mathfrak{b}\) (and therefore \(b\in B_0\)), there exists an ultrafilter
  convergence \(\mathfrak{a}\leadsto a\) in \(A\) with \(A_0\in \mathfrak{a}\)
  (and therefore \(a\in A_0\)) which is mapped to \(\mathfrak{b}\leadsto b\),
  that is \( f_0 \) is a descent morphism in \( \Top \).
\end{example}

\begin{example}
  When we equip \(X=[0,1]\) with the lower topology as in Example
  \ref{mm:examples1}, for every ultrafilter \(\mathfrak{x}\) on \([0,1]\) (see
  \cite{Hof07}, for example),
  \begin{displaymath}
    \sigma(\mathfrak{x})=
    \sup\{u\in [0,1]\mid {\uparrow}u\in \mathfrak{x}\}.
  \end{displaymath}
  Hence, by Theorem~\ref{d:thm:8}, a morphism \(f \colon (A,(A_u)_{u\in
  [0,1]})\to (B,(B_u)_{u\in [0,1]})\) in \(\LaxComma{\Top}{[0,1]}\) is
  effective for descent if and only if \(f \colon A\to B\) is effective for
  descent in \(\Top\) and, for all \(u\in [0,1]\) and all ultrafilter
  convergences \(\mathfrak{b}\leadsto b\) in \(B_u\), for every \(v<u\) there
  exists an ultrafilter convergence \(\mathfrak{a}\leadsto a\) in \(A_v\) with
  \(Uf(\mathfrak{a})=\mathfrak{b}\) and \(f(a)=b\).
\end{example}

The interval \( [0,1] \) is a completely distributive lattice, and its totally
below relation \( \lll \) is given by \( < \). The arguments used for \( [0,1]
\) generalise smoothly to completely distributive lattices, as stated in the
following result.
\begin{theorem}
  Let \(X\) be a completely distributive lattice equipped with the lower
  topology. A morphism \(f \colon (A,\alpha)\to(B,\beta)\) is effective for descent
  in \(\TopX\) if and only if
  \begin{enumerate}
    \item
      \( f \colon A \to B \) is effective for descent in \( \Top \);
    \item
      for each \(u\in X\), all ultrafilter convergences \(\mathfrak{b}\leadsto
      b\) in \(B_u\), and each \(v\lll u\), there exists an ultrafilter
      convergence \(\mathfrak{a}\leadsto a\) in \(A_v\) with
      \(Uf(\mathfrak{a})=\mathfrak{b}\) and \(f(a)=b\).
  \end{enumerate}
\end{theorem}



\end{document}